\documentclass[12pt]{article}
\usepackage[left=2cm, right=2cm, top=2.5cm, bottom=2.5cm]{geometry}

\usepackage[latin,english]{babel}
\usepackage{comment}
\usepackage[utf8]{inputenc}
\usepackage[T1]{fontenc}
\usepackage{eufrak, color, mathrsfs,dsfont,geometry, bbm}
\usepackage{amsmath, amssymb, amsthm, pdfsync, hyperref}
\usepackage{mathtools, cases, braket, upgreek}
\usepackage{babelbib}
\usepackage{math}
\mathtoolsset{showonlyrefs}

\numberwithin{equation}{section}

\theoremstyle{plain}
\newtheorem{theorem}{Theorem}[section]
\newtheorem{lemma}[theorem]{Lemma}
\newtheorem{proposition}[theorem]{Proposition}

\theoremstyle{definition}
\newtheorem{definition}[theorem]{Definition}

\theoremstyle{remark}
\newtheorem{remark}[theorem]{Remark}

\renewcommand{\epsilon}{\varepsilon}
\newcommand{\Opw}{\operatorname{Op^w}}

\begin{document}
\title{Long-time behavior of resonant time-dependent perturbations of periodic transport equations on $\R$}

\author{Maria Teresa Rotolo\footnote{	International School for Advanced Studies (SISSA), Via Bonomea 265, 34136, Trieste, Italy. \newline 
	 \textit{Email:} \texttt{mrotolo@sissa.it}}}
\date{}

\maketitle

\begin{abstract}
We consider linear, time-dependent and skew-adjoint perturbations of periodic transport equations on the one-dimensional torus. We describe the long-time behavior of solutions for all non-degenerate perturbations in resonant regime, proving that either there exist solutions whose Sobolev norms explode exponentially fast, provoking energy transfer phenomena, or all solutions remain stable for arbitrarily long time scales. The proof combines pseudodifferential tools with dynamical systems results:  we perform a resonant normal form procedure to reduce our  analysis to the classical dynamics for the resonant equation. The main difficulty lies in the proof of the instability result, for which we explicitly construct an escape function associated to the dynamics. This is obtained by means of a positive commutator estimate on the operator associated with the escape function, exploiting microlocal analysis. 
\end{abstract}

\tableofcontents

\section{Introduction and main result}

In this article we study long-time behavior of solutions for a class of linear and time-dependent transport equations on the one-dimensional torus. Two phenomena that typically occur both in linear and nonlinear PDEs are long-time stability of solutions or existence of unstable orbits, and a valid tool to depict such behaviors is the study of the Sobolev norms \eqref{sobolev.spaces} of solutions. On the one hand, an unbounded growth in time of the positive Sobolev norms of some solutions is an indicator of instability, in the form of energy transfer phenomena; on the other, stable equations are characterized by uniform-in-time bounds for all Sobolev norms of all solutions, on long time scales. 
The equations that we consider are small in size, $2\pi$-periodic in time perturbations of constant coefficient transport equations. Precisely, 
\begin{equation}\label{initial.transp.eq}
    \pa_t u= (m + \e V(t, x))\pa_x u + \frac{\e}{2} \pa_x V(t,x) u, \quad x \in \T := \R / 2\pi \Z, \ t \in \R,
\end{equation}
where $m \in \N^+$, $\e$ is a positive parameter, $V \in C^{\infty}(\T^2; \R)$. Observe immediately that the right-hand side of \eqref{initial.transp.eq} is skew-adjoint over $L^2(\T)$, thus one has, for every solution of \eqref{initial.transp.eq},
\begin{equation}
    \norm{u(t)}_{L^2(\T)}= \norm{ u(0)}_{L^2(\T)}, \quad \forall t \in \R. 
\end{equation}
Moreover, defining the scale of Sobolev spaces $\cH^s$, $s \in \R$ on the torus as 
\begin{equation}\label{sobolev.spaces}
    \cH^s(\T, \C)=\cH^s:=\{ u \in L^2(\T) : \norm{u}_s^2:=\sum_{k \in \Z} \la k \ra^{2s} |u_k|^2 < +\infty\}, 
\end{equation}
(where we use the notation $\la k \ra:=(1+|k|^2)^{\frac12}$ for the Japanese brackets and we denote by $\{u_k\}_{k \in \Z}$ the Fourier coefficients of the function $u$), one has that, for $\e=0$, all Sobolev norms of all solutions of \eqref{initial.transp.eq} are constant in time:
\begin{equation}
    \norm{u(t)}_s=\norm{u(0)}_s, \quad \forall s \in \R, t \in \R.
\end{equation}
This property is anyway not stable under perturbation: as soon as $\e \neq 0$, the presence of a perturbation can lead to unbounded paths. 

In the recent work \cite{Riviere.Rotolo}, the authors consider equation \eqref{initial.transp.eq} on $\T^n$ for any $n \in \N^+$ and prove that there exists an open set of time-dependent perturbations $V \in C^\infty(\T \times \T^{n})$ that provoke unstable solutions (where, by unstable solutions, we intend solutions exhibiting unbounded growth in time of their positive Sobolev norms). Moreover, this set also turns out to be dense when $n=2$. In the present work, we consider instead the one-dimensional case and we analyze the presence of stable and unstable solutions for all time-dependent perturbations $V$ having non degenerate resonant average (see Definitions \ref{non.deg}, \ref{resonant.average}). We anticipate here that such set turns out to be dense in $C^\infty(\T^2; \R)$. 
Roughly speaking we prove that, for such perturbations, only two possibilities occur: either all Sobolev norms of all solutions are uniformly bounded for long time scales, or instability phenomena arise. 

We now give the preliminary definitions needed to state our main result. 

\begin{definition}[Resonant average]\label{resonant.average}
    Let $V \in C^{\infty}(\T^2; \C)$.
    For $k, \ell \in \Z$ denote by $\{v_{k, \ell}\}_{k, \ell \in \Z}$ the Fourier coefficients of $V$, i.e., 
$
V(t, x)= \sum_{k, \ell \in \Z} v_{k, \ell}e^{i (kx+\ell t)}
$, 
for all  $t, x$ in $\T$. 
We define the \emph{resonant average} of $V$ with respect to the frequency $m \in \N^+$ as 
\begin{equation}\label{1d.average}
\la V \ra_m(x):= \sum_{k \in \Z} v_{k, mk} e^{ikx}.
\end{equation}
We say that a function $V(t, x)$ is \emph{completely resonant} with respect to the frequency $m \in \N^+$ if $v_{k, \ell}=0$ for all $\ell \neq m k$, or equivalently if $V(t, x) \equiv\la V \ra_m (x+ mt)$.
\end{definition}

\begin{remark}
One can verify that \eqref{1d.average} coincides with $ \la V \ra_m(x):= \frac{1}{2\pi} \int_{\T} V(t, x- m t) \di t$,
which is the time average of $V$ along the unperturbed flow (i.e., the flow of \eqref{initial.transp.eq} with $\e=0$).  
\end{remark}

\begin{definition}\label{non.deg}
    We say that a function $W \in C^\infty(\T; \C)$ has a degenerate zero in $x_0 \in \T$ if $W(x_0)=0$ and $W'(x_0)=0$. We say that a function is non degenerate if it does not have any degenerate zero. 
\end{definition}

We also introduce the following terminology:
\begin{definition}\label{res.non.deg}
We say that a function $V \in C^\infty(\T^2; \R)$ is
\begin{itemize}
    \item \emph{resonantly stable} with respect to the frequency $m \in \N^+$ (or simply  resonantly stable) if its resonant average $\la V \ra_m$ (see Definition \ref{resonant.average}) does not have any zero;
    \item \emph{resonantly unstable} with respect to the frequency $m \in \N^+$ (or simply resonantly unstable)
 if its resonant average $\la V \ra_m $ (see Definition \ref{resonant.average}) has at least one zero, but all its zeroes are non degenerate. 
 \end{itemize}
\end{definition} 

Our main result is the following theorem. 

\begin{theorem}\label{main.thm}
For any $m \in \N^+$ let
\begin{equation}\label{Vs}
    \cV_{m, s}:= \{ V \in C^\infty(\T^2; \R) : V \mbox{ is resonantly stable } \}, 
\end{equation}
\begin{equation}\label{Vu}
    \cV_{m, u}:= \{ V \in C^\infty(\T^2; \R) : V \mbox{ is resonantly unstable } \},
\end{equation}
\begin{equation}\label{Vd}
    \cV_{m, d}:= \{ V \in C^\infty(\T^2; \R) : \la V\ra_m \mbox{ has  degenerate zeroes } \},
\end{equation}
where we recall Definitions \ref{resonant.average}, \ref{non.deg} and \ref{res.non.deg}. The following holds:
\begin{enumerate}
    \item \b{Long-time stability:} for any $V \in \cV_{m, s}$, for any $N > 0$ there exists $\e_N(V)>0$ such that, for any $0< \e < \e_N(V)$ and for any $s>0$,  there exists $C_1>0$ (independent of $\e$) such that, for any $u_0(x) \in \cH^s(\T)$, the corresponding solution $u(t,x)$ of equation \eqref{initial.transp.eq} (i.e. $u(0, x)=u_0(x)$) satisfies 
    \begin{equation}\label{stable.norms}
        \norm{u(t)}_s \leq C_1 \norm{u_0}_s \quad \forall \, |t| \leq C_V \e^{-(N+1)}, \, \mbox{ where } \ C_1 \equiv C_1(s, V)
    \end{equation}
    for some $C_V>0$.
    
    \item \b{Instability:} for any $V \in \cV_{m, u}$ there exist $\e(V)>0$, $\delta_1, \delta_2>0$ such that, for any $0< \e < \e(V)$  and for any $s>1$ there exists $u_0 \in \cH^s(\T)$,  such that the corresponding solution $u(t, x)$ of equation \eqref{initial.transp.eq} satisfies: 
    \begin{equation}\label{cond.growth}
        \norm{u(t)}_s \geq \delta_1 e^{\e \delta_2 t} \norm{u_0}_s, \quad \forall \ t >0\,.
    \end{equation}

    \item \b{Genericity:} The set $\cV_{m, u} \cup \cV_{m, s}$ is open and dense in $C^\infty(\T^2;\R)$ endowed with its Frechét topology. 
\end{enumerate}
\end{theorem}

First of all remark that
\begin{equation}\label{split.C.infty}
    C^\infty(\T^2; \R)= \cV_{m, s} \cup \cV_{m, u} \cup \cV_{m, d}
\end{equation}
(see \eqref{Vs},\eqref{Vu}, \eqref{Vd}). Thus, from the genericity property in point 3, Theorem \ref{main.thm} answers the question of stability of equation \eqref{initial.transp.eq} for the vast majority of time-dependent, $2\pi$-periodic perturbations $V$. 
Moreover, the lower bound in \eqref{cond.growth} is sharp: in \cite{MasperoRobert} the authors prove that all positive Sobolev norms of all solutions of \eqref{initial.transp.eq} grow at most exponentially fast. We remark anyway that our bound \eqref{cond.growth} is effective for $t \gtrsim \e^{-1}$. 

Next, let us motivate the definitions of the sets $\eqref{Vs}$, $\eqref{Vu}$ and $\eqref{Vd}$. Remark that the perturbation $V$ in \eqref{initial.transp.eq} has the same time-period as the solution to the unperturbed equation (i.e., $\e=0$): this is what we call \emph{resonant regime}. For this reason, we look at the dynamical properties of  $\la V \ra_m$ (see \eqref{1d.average}) in order to study long-time behavior of equation \eqref{initial.transp.eq}: roughly speaking, $\la V \ra_m$ collects only the resonant modes of $V$, that are the ones responsible for the formations of instabilities.

However, not all potentials having non-zero resonant average generate instabilities: the existence of at least a non-degenerate zero (see Definition \ref{non.deg}) of the resonant average is a necessary condition. As we will explain in detail in the next paragraph, this is linked to the dynamics of the system $\dot{x} = \la V \ra_m(x)$, for which such zero corresponds to a hyperbolic critical point. 

On the other hand, the presence of instabilities is a typical feature of resonant regimes that has recently been studied for different systems: 
\begin{itemize}
    \item[-] as anticipated, in the recent work \cite{Riviere.Rotolo}, instabilities for the higher-dimensional analogue of equation \eqref{initial.transp.eq} are studied. Also in the higher dimensional case, the result relies on the study of the dynamics of $\dot{x}= \la V \ra_m(x)$, with $x \in \T^n$, $n \geq 2$, which clearly is much more involved than the one-dimensional system considered in the present work. We refer to \cite{Riviere.Rotolo} and  \cite{DR.MS1} for a detailed description of this dynamics, and we only emphasize here that the key feature exploited is hyperbolicity of the system and presence of closed, disjoint attracting and repelling sets. This is the same property that we use in the present work for the one dimensional case, where it is an evident consequence of the structure of the function $\la V \ra_m$ (see Section \ref{sec.escape}).  
    \item[-] genericity of instability has also been obtained for harmonic oscillators in dimension one in \cite{Mas23} and in dimension two in \cite{LangellaMasperoRotolo25} . In both cases, the results come as an application of \cite{Mas22}, where the author identifies sufficient conditions on the perturbation that ensure presence of unstable solutions. Again the dimensionality of the problem plays a central role: in the one dimensional case, the study of the associated classical dynamics can be carried out in a quite direct way, tailored for the specific problem.  The step of raising the dimension is instead tackled in \cite{LangellaMasperoRotolo25} and requires significantly different and more involved techniques for the study of the dynamics, that comprehend arguments of contact geometry and differential topology.
\end{itemize}

Let us finally compare with the \emph{non-resonant} regime: in the recent paper \cite{BLM.red.T} the authors consider transport terms of the form  $\nu + \e V(\omega t, x)$, with $(\nu, \omega) \in \R^n \times \R^d$ and $V \in C^\infty(\T^n \times \T^d)$. In this setting, they are able to use KAM techniques to prove reducibility of the equation for a non-resonant choice of the parameters $(\nu, \omega) \in [1, 2]^{n+d}$, belonging to a suitably chosen Cantor set (see also \cite{BBHM2018}, \cite{BBM2014}, \cite{BGMR2017}, \cite{FeolaGiulianiMontaltoProcesi} for related results). Roughly speaking, they find a bounded transformation (on the Sobolev scale $\cH^s$ \eqref{sobolev.spaces})  that conjugates the equation to a diagonal one in the Fourier basis, up to smoothing remainders. A direct consequence of this result is that either there exist solutions whose positive Sobolev norms blow up exponentially (in the case where the diagonal operator has eigenvalues with nonzero real parts) or the Sobolev norms of all solutions remain bounded for all times (see \cite[Corollary 2.5]{BLM.red.T} for details). 

\subsection{Outline of the Proof}
In this section we quickly describe the main tools needed to prove Theorem \ref{main.thm}.

Even though the strategies to prove the long-time stability and the instability result are quite different, the starting point is the same: for every smooth perturbation $V \in C^\infty(\T^2; \R)$ we are able to conjugate the initial equation \eqref{initial.transp.eq} to a new transport equation, having as leading coefficient of the transport operator the resonant average $\la V \ra_m$ of $V$ (see Definition \ref{resonant.average}). We refer to Lemma \ref{lemma.normal} for a precise statement of this result. An analogous procedure has already been used in \cite{Riviere.Rotolo}, but we remark here that, thanks to the one dimensional setting of the present work, we are able to treat both stability and instability starting with this initial conjugation, whilst in \cite{Riviere.Rotolo} the procedure is tailored to the instability result. 

We now briefly describe the strategy to prove long-time stability: the core idea of our proof is that, when considering $V \in \cV_{m,s}$ (see \eqref{Vs}), the above-mentioned conjugation procedure results in a transport equation with never vanishing transport term, up to arbitrarily small-in-size remainders (precisely, it is possible to obtain remainders of order $\e^N$ for every $N \in \N^+$). This enables us to reduce the equation to a constant-coefficient transport equation, again up to arbitrarily small-in-size non-constant remainders. Roughly speaking, the constant  coefficient is obtained integrating the inverse of the resonant average (see Proposition \ref{prop.const.coeff} for the precise procedure). For this reason, we crucially need to have $\la V\ra_m \in \cV_{m, s}$. Finally, we remark that, due to the presence of the remainder term, we can not hope to obtain global in time stability. Anyway, the time scale for which stability holds can be chosen arbitrarily long, up to reducing the size of the perturbation accordingly.

We now turn to the  description of the proof of the instability result. In this case, we use a different approach: we use the information on the dynamics  of the ODE $\dot{x}= \la V \ra_m(x)$ to analyze the associated dynamics of the Hamiltonian vector field $\cX_h$ on $T^*\T \simeq \T \times \R$, with Hamiltonian $h(x, \xi): = \xi \la V\ra_m (x)$. As one can immediately see from the Hamilton equations (see \eqref{Xh}), the two dynamics are strictly related and, in particular, the key point of our proof is the construction of an \emph{escape function} for the Hamiltonian dynamics.  Roughly speaking, it is a function that increases along the flow of $\cX_h$. Next we use pseudodifferential calculus and microlocal tools such as G\aa{}rding inequality to recover the exponential growth in \eqref{cond.growth} via a positive commutator estimate (see Section \ref{sec.energy.est}). 

Anyway, we anticipate here that the construction of the escape function can be far from trivial, and strongly depends on the topological setting of the Hamiltonian system, and thus in turn on the spacial domain of the initial PDE. In the present work we adapt to our context the constructions of \cite{Colin_de_Verdi_re_2020} and \cite{LangellaMasperoRotolo25}, and we show that, in this one dimensional periodic case, it is sufficient to consider smooth vector fields
having only (but at least one) hyperbolic critical points. We briefly mention that, in higher dimensional contexts, it is no more sufficient to ask only for the existence of critical elements of the vector fields, but a central role is played by the hyperbolicity assumptions, that are in such cases much more involved: we only cite here \cite{Faure.Sjostrand}, where the microlocal point of view to treat similar types of problems is introduced for Anosov vector fields and the recent book \cite{Lefeuvre} and the references therein, for an introduction to the geometrical tools needed in higher dimension.

\subsection{Comparison with related literature}
In this last introductory section we comment about the existing literature on growth of Sobolev norms for linear, time-dependent perturbations of Schrödinger equations.

The general problem of growth of Sobolev norms as detector of instabilities has been addressed from the perspective of quantum mechanics in \cite{Bellissard} already in the 80's. For what concerns linear, time-dependent Schrödinger equations on the torus, the problem has been first issued by Bourgain in \cite{Bourgain1999} (see also the more recent work \cite{Chabert}) and in the recent years many authors have proved existence of particular time-dependent perturbations of linear Schrödinger equations, that generate unbounded orbits. 

First, unstable perturbations of the one dimensional Harmonic oscillators are identified in \cite{Delort}, \cite{Mas22}, and in \cite{Mas23}, where a generic class of unstable perturbations is found (see also \cite{BGMR}, \cite{Mas18},\cite{FaouRaphael}, \cite{LangellaMasperoRotolo25}). 
We also refer to the recent result \cite{LLZ} for the particular case of pseudodifferential perturbations which are quadratic polynomials both in $x$ and $D$ and we emphasize that the growth result obtained holds in \emph{any dimension}. Regarding related models, we report \cite{BGMRV} for perturbations of the Landau Hamiltonian, \cite{HausMaspero} for anharmonic oscillators and \cite{Thomann} for analogous results in the Bargmann-Fock space.

Taking the moves from the abstract result in \cite{Mas22}, in which the author uses for the first time Mourre's positive commutator theory to perform the key energy estimate for the growth result, in the recent works \cite{Mas23}, \cite{LangellaMasperoRotolo25}, \cite{Riviere.Rotolo}, this mechanism is exploited in different models. Precisely, \cite{LangellaMasperoRotolo25} is inspired by the construction of the escape function performed in \cite{CdVSR}, \cite{Colin_de_Verdi_re_2020}, that is in turn adapted also in the present work. On the other hand, \cite{Riviere.Rotolo} deals with the more involved construction of the escape function for the higher dimensional analogue of equation \eqref{initial.transp.eq}.

Finally, constructing solutions with unbounded orbits in nonlinear Schrödinger-like equations is an open challenge in this field (see, among others, \cite{Kuk}, \cite{Kuk2}, \cite{CKSTT}, \cite{gerard_grellier}, \cite{GL}, \cite{Chabert2}, \cite{guardia_kaloshin}, \cite{hani14}, \cite{hani15}, \cite{GG}, \cite{GHMMZ}). We mention separately the recent works \cite{MasperoMurgante} for an application of the escape function-method to a quasi-linear case and \cite{Vlasov} where similar strategies are used to study the exponential stability of nonlinear Vlasov equations on negatively curved manifolds.

\paragraph{Organization of the paper}
We conclude this introductory section by quickly outlining the structure of the paper. In Section \ref{sec.preliminary} we collect preliminary results needed along the work. In particular, in Section \ref{sec.pseudo} we introduce the class of pseudodifferential operators along with some of their properties, whilst in Section \ref{sec.conj} we define the class of transformations that we will use all along the work. 
In Section \ref{sec.stability} we prove the first point of Theorem \ref{main.thm}, using first a normal form procedure to reduce equation \eqref{initial.transp.eq} to constant coefficients up to small-in-size remainders, and then proving estimate \eqref{stable.norms}. 

In order to prove the second point of Theorem \ref{main.thm}, as already anticipated, we need to construct an escape function for the Hamiltonian lift of the flow generated by the resonant average of $V$ on the torus. Thus in Section \ref{sec.escape} we construct the escape function and in Section \ref{sec.energy.est} we use it to prove estimate \eqref{cond.growth}.
Finally, in Section \ref{sec.generic} we conclude the proof of Theorem \ref{main.thm} studying the structure of the sets $\cV_{m,s}, \cV_{m,u}$.

\paragraph{Acknowledgements.} The problem studied in the present article has been pointed out by A. Maspero and B. Langella, that the author warmly thanks for the idea and the support during the preparation of this manuscript. The author also wishes to thank G. Rivière for the useful discussions. 
The author acknowledges the support of the INdAM group GNAMPA.

\section{Preliminary results}\label{sec.preliminary}
In this section we collect some preliminary results that we will use in the rest of the work. Precisely, in Section \ref{sec.pseudo} we collect some definitions and results of symbolic calculus and in Section \ref{sec.conj} we define a class of linear transformations (that we will need to prove both the reducibility and the instability results in Theorem \ref{main.thm}) and we study their action on transport operators.

\subsection{Pseudodifferential operators}\label{sec.pseudo}
We define the class of symbols of order $\rho \in \R$ as
\begin{equation}\label{symbol}
     \bS^{\rho}(\T):=\{ a \in C^\infty(\T \times \R) :  \forall \alpha, \beta \in \N^n \ |\pa_x^{\alpha}\pa_{\xi}^{\beta} a(x, \xi) | \leq C_{\alpha, \beta} \la \xi \ra ^{\rho -|\beta|}\}, \quad \forall \rho \in \R. 
\end{equation}
We also denote by $C^\infty(\T; \bS^{\rho})$ the space of smooth maps $t \mapsto a(t, \cdot) \in \bS^\rho(\T)$, for every $\rho \in \R$. For every $a \in \bS^{\rho}(\T)$ we define the Weyl quantization $\Opw(a)$ of $a$ as 
\begin{equation}\label{weyls.quantization}
    \Opw(a)[u](x):= \frac{1}{2\pi} \int_{\T \times \R} e^{i(x-y)\xi}a\left(\frac{x+y}{2}, \xi\right)u(y)\di y \di \xi. 
\end{equation}
We say that an operator $A$ is pseudodifferential of order $\rho$ if there exists $a \in C^\infty(\T\times \R)$ such that $A=\Opw(a)$. We write $A \in \cS^{\rho}$ and we denote by $C^\infty(\T; \cS^\rho)$ the set of smooth maps $t \mapsto A(t) \in \cS^\rho$, for every $\rho \in \R$. 
We now collect in the following lemma the main results of symbolic calculus that we will need in the work. We refer for example to \cite{Shubin}, \cite{Zworski} for the proofs.\\

\b{Notation:} along the article we sometimes write mod $\cS^\rho$, $\rho \in \R$, next to an operator, to indicate that it is defined up to pseudodifferential remainders of order $\rho$. Analogously, we write mod $\bS^\rho$ next to symbols. We will also use the notation $\norm{\cdot}_{\cH_0 \to \cH_1}$ to denote the operatorial norm of the argument as operator from the Hilbert space $\cH_0$ to the Hilbert space $\cH_1$. 

\begin{lemma}\label{symbolic.calculus}
    Let $f \in \b{S}^{\rho_1}(\T)$, $g \in \b{S}^{\rho_2}(\T)$, $\rho_1, \rho_2 \in \R$. Then:
    \begin{enumerate}
        \item For any $s \in \R$, there exist $C_{f, s}>0$ such that: 
        \begin{equation}\label{e.cv}
          \|\Opw(f)u\|_{s-\rho_1} \leq C_{f, s}\|u\|_s, \quad \forall u \in \cH^s,
        \end{equation}  
        where we recall the definition of Sobolev norms in \eqref{sobolev.spaces};
        \item
        $\Opw(f)^*=\Opw(\overline{f})$ (here $^*$ denotes the adjoint operator) and in particular skew-adjoint operators have purely imaginary symbols;
        \item There exists $h \in \b{S}^{\rho_1+\rho_2}(\T)$ such that 
        $\Opw(f)\circ \Opw(g)=\Opw(h).$ Furthermore,
        \begin{equation}\label{comp.symbols}
            h:= f g -\frac{i}{2} \{ f, g\} \ \operatorname{mod}\ \b{S}^{\rho_1+\rho_2-2}(\T); 
        \end{equation}
        \item 
        There exists $\ell  \in \b{S}^{\rho_1+\rho_2-1}(\T)$ such that 
        $ i [\Opw(f),\Opw(g)] = \Opw(\ell)$. 
        Moreover
        \begin{equation}\label{comm.symbols}
        \ell:= \{ f, g\}\ \operatorname{mod}\ \bS^{\rho_1+ \rho_2-2}(\T);
        \end{equation}
        \item
        If $\rho_1<0$, then $\Opw(f)$ is compact as an operator from $L^2(\T)$ to itself. 
        \item If $a \in C^\infty(\T \times \R)$ is a compactly supported function, then $a \in \bS^{\rho}$ for any $\rho<0$. 
    \end{enumerate}
\end{lemma}

\begin{remark}\label{op.0and1}
    Let $p \in C^\infty(\T)$. Then $p \in \bS^0$ and $\Opw(p)[u](x)= p(x)u(x)$. Moreover
    \begin{equation}\label{id.diff.1}
        \Opw(i\xi p(x))= p(x) \pa_x + \frac12 p'(x). 
    \end{equation}
\end{remark}
\noindent Finally, we state the following G\aa rding inequality, referring to \cite[Theorem 9.11]{Zworski} for a proof.  
\begin{lemma}
    Let $a \in \bS^{\rho}(\T)$, $a \geq 0$. Then for every $u \in C^\infty(\T)$ we have
    \begin{equation}\label{eq.garding}
        \la \Opw(a) u, u \ra_{L^2(\T)} \geq - C_{\rho, a} \norm{u}^2_{\frac{\rho-1}{2}},
    \end{equation}
    for some $C_{\rho, a}>0$. 
\end{lemma}

\subsection{Conjugation results}\label{sec.conj}
In this section we define a class of linear, invertible transformations and study their action on transport equations. This conjugation technique is already used in \cite{BLM.red.T} and \cite{FeolaGiulianiMontaltoProcesi}. 

The same transformation as the one presented in this section is studied in  \cite[Section 5.1]{Riviere.Rotolo}, for tori of any dimension. We report it here since we slightly modify the discussion to adapt it also to the proof of the stability result (see Section \ref{sec.stability}).

First of all we remark that a linear, invertible transformation $u=A(t)u_1$, with $A(t)$ periodic in time of period $2\pi$, transforms equation $\dot{u}=Hu$ into $\dot{u}_1=A(t)_*H u_1$, where the pushforward $A(t)_*H$ is given by 
\begin{equation}\label{pushforward}
A(t)_*H:=A(t)^{-1}[H A(t) -\pa_t A(t)].
\end{equation}
We remark, to avoid any confusion in the notation, that $\pa_t A(t)$ is the derivative in time of the linear transformation (and not the composition $\pa_t \circ A(t)$). Next, we consider a family of time-dependent diffeomorphisms of the torus of the form 
\begin{equation}\label{diffeo}
\varphi_k(t): \T \to \T, \quad x \mapsto x +\e^k \beta(t, x), \quad k \in \N^+, \ \e>0, \ \beta \in C^\infty(\T^2; \R).
\end{equation}
 It has been proved (see for example \cite[Lemma B.4]{Baldi}) that for $\e>0$ small enough $\varphi_k(t)$ is an invertible diffeomorphism and its inverse has the form 
\begin{equation}\label{inv.diffeo}
   \varphi^{-1}_k(t): \T \to \T, \quad y \mapsto y + \e^k \tilde{\beta}(t, y),  \quad k \in \N^+, \e>0,  
\end{equation}
for some $\tilde{\beta} \in C^{\infty}(\T^2; \R)$. We define the transformation associated with \eqref{diffeo} as: 
\begin{equation}\label{Phi}
   \tilde{\Phi}_k(t): u(x) \mapsto (1+ \e^k \pa_x \beta) ^{\frac12}  (u \circ \varphi_k) (x) ,  \quad \forall t \in \T, k \in \N^+, 
\end{equation}
\noindent with inverse given by 
\begin{equation}\label{inv.Phi}
\quad \tilde{\Phi}_k(t)^{-1}: u(y) \mapsto (1 + \e^k \pa_x \tilde{\beta}) ^{\frac12} (u \circ \varphi_k^{-1})(y), \quad \forall t \in \T, k \in \N^+.
\end{equation}
We stress here that $\tilde{\Phi}_k(t) \equiv \tilde{\Phi}_{k, \e, \beta}(t)$, i.e., the transformation depends both on $\beta$ and $\e$, but we will always write as in \eqref{Phi}, omitting such labels to simplify the notation. \\
Remark that, for every $k \in \N^+$, both $\tilde{\Phi}_k(t)$ and $\tilde{\Phi}_k(t)^{-1}$ are $2\pi$-periodic in time. We now show that, for every $k \in \N^+$, the operator $\tilde{\Phi}_k(t)$ in \eqref{Phi} is unitary: substituting the explicit expressions of \eqref{diffeo} and \eqref{inv.diffeo} in the identity
\begin{equation}\label{compose.inverse}
    \varphi_k\vert_{x= y + \e^k \tilde{\beta}(t,y)} \circ \varphi_k^{-1} (y)= y \qquad \forall t \in \R, \ y \in \T, \ k \in \N^+,
\end{equation}
and deriving \eqref{compose.inverse} with respect to $y$ one  obtains that, for every $\e>0$ and $\beta \in C^\infty(\T^2; \R)$,  
\begin{equation}\label{inv.det}
    1+ \e^k \pa_y \tilde\beta(t,y)= \frac{1}{1 +\e^k \pa_x \beta(t,x)\vert_{x=y + \e^k \tilde{\beta}(t, y)}}, \quad \forall y \in \T, \ k \in \N^+.
\end{equation}
This in turn gives, via a direct computation, unitarity of $\tilde{\Phi}_k(t)$, namely
\begin{equation}\label{unitary}
\tilde{\Phi}_k(t)^*=\tilde{\Phi}_k(t)^{-1}, \quad \forall t \in \T, k \in \N^+,
\end{equation}
where $\tilde{\Phi}_k(t)^*$ denotes the adjoint operator of $\tilde{\Phi}_k(t)$. We now turn to studying the action of $\tilde{\Phi}_k(t)$ on transport operators.

\begin{lemma}\label{preliminary.o1}
Take any $f$, $g \in C^\infty(\T^2; \R)$, $n \in \N^+$ and $\e>0$. Define 
\begin{equation}\label{H.n}
     H_n(t):= \Opw(i \xi (f(t,x) + \e^n g(t,x)) ).
    \end{equation}
     For any $\beta \in C^\infty(\T^2; \R)$ there exists $\e_\beta>0$ such that, for every $0< \e < \e_\beta$, the associated transformation $\tilde{\Phi}_n(t)$ defined as in \eqref{Phi} satisfies 
    \begin{equation}\label{exp.one1}
       \tilde{\Phi}_n(t)^{-1}H_n(t) \tilde{\Phi}_n(t) = \Opw(i\xi(f(t,x) +\e^n (g(t,x)+ f(t,x)\pa_x \beta(t,x))+ \e^{n+1} \cR^1_{n, \e}(t,x)))
    \end{equation}
    and  
    \begin{equation}\label{exp.one2}
        \tilde{\Phi}_n(t)^{-1}\pa_t \tilde{\Phi}_n(t)= \Opw(i\xi(\e^n \pa_t \beta(t,x) + \e^{n+1} \cR^2_{n, \e}(t,x))),
    \end{equation}
    where $\cR^{1}_{n, \e}$, $ \cR^{2}_{n, \e} \in C^\infty(\T^2;\R)$ have all seminorms uniformly bounded in terms of $\e \in [0, \e_\beta]$. 
\end{lemma}

\begin{proof}
    Given $\beta \in C^\infty(\T^2;\R)$, let $\tilde{\Phi}_n(t)$ be the associated transformation in \eqref{Phi}, with $k \leadsto n$ and recall the expression for its inverse in \eqref{inv.Phi}. Fix $\e_\beta>0$ such that, for every $0 < \e < \e_\beta$,   $\tilde{\Phi}_n(t)$ is a bounded and invertible transformation  (recall previous discussion in \eqref{diffeo} and \eqref{inv.diffeo} and see \cite[Lemma B.4]{Baldi}). We first prove \eqref{exp.one1}. To this aim, first remark that, since $H_n(t)$ in \eqref{H.n} is skew-adjoint (see Lemma \ref{symbolic.calculus}), using \eqref{unitary} we have 
\[
[\tilde{\Phi}_n(t)^{-1}H_n(t)\tilde{\Phi}_n(t) ]^* \stackrel{\eqref{unitary}}{=}-\tilde\Phi_n(t)^{-1}H_n(t)\tilde{\Phi}_n(t).
\]
Next, using Lemma \ref{symbolic.calculus}, we can rewrite $H_n(t)$ in \eqref{H.n} as
\begin{equation}\label{Hn.diff}
    H_n(t)= (f(t,x) + \e^n g(t,x))\pa_x + \frac12\pa_x(f(t,x) + \e^ng(t,x)).
\end{equation}
Thus, using \eqref{Hn.diff}, \eqref{Phi}, \eqref{inv.Phi}, \eqref{inv.det} and performing a direct computation one obtains 
\begin{equation}\label{compute1}
\begin{split}
     \tilde{\Phi}_n(t)^{-1}H_n(t) \tilde{\Phi}_n(t) & = (f + \e^n g)(t, \varphi_n^{-1}(t)[x]) (1 + \e^n \pa_x \beta(t, \varphi^{-1}_n(t)[x])) \pa_x \\
     &+ \frac12\pa_x \left((f + \e^n g)(t, \varphi_n^{-1}(t)[x]) (1 + \e^n \pa_x \beta(t, \varphi^{-1}_n(t)[x]))\right)\\
     & \stackrel{\eqref{comp.symbols}}{=} \Opw( i\xi((f + \e^n g)(t, \varphi_n^{-1}(t)[x]) (1 + \e^n \pa_x \beta(t, \varphi^{-1}_n(t)[x]))).
     \end{split}
\end{equation}
Analogously, the expression for $\tilde\Phi_n(t)^{-1}\pa_t \tilde\Phi_n(t)$ can be explicitly computed and is given by 
\begin{equation}\label{compute2}
\begin{split}
    \tilde\Phi_n(t)^{-1}\pa_t \tilde\Phi_n(t) & = \e^n \pa_t \beta(t, \varphi^{-1}_n(t)[x]) \pa_x + \frac{\e^n}{2}\pa_x \left(\pa_t \beta(t, \varphi^{-1}_n(t)[x]) \right) \\
    & \stackrel{\eqref{comp.symbols}}{=} \Opw(i\xi\pa_t\beta(t, \varphi^{-1}_n(t)[x])). 
\end{split}
\end{equation}

Finally, performing a Taylor expansion of both \eqref{compute1} and \eqref{compute2} in a neighborhood of $\e=0$ we obtain respectively \eqref{exp.one1} and \eqref{exp.one2} concluding the proof. 
\end{proof}

We conclude this preliminary section with the following Lemma, which contains a computation that we will use repeatedly along the work. 

\begin{lemma}\label{preliminary.find.beta}
Let $m \in \N^+$ and $V \in C^\infty(\T^2; \R)$. There exists a smooth function $\beta_{m, V} \in C^\infty(\T^2; \R)$ such that 
\begin{equation}\label{find.beta}
    V(t,x) + m \pa_x\beta_{m, V}(t,x) - \pa_t \beta_{m, V}(t,x)= \la V \ra_m(x + mt)
\end{equation}
(recall the definition \eqref{1d.average} of resonant average of $V$ with respect to $m$). 
\end{lemma}

\begin{proof}
    Denote by $T(t, x):=V + m \pa_x\beta_{m, V} - \pa_t \beta_{m, V}$, with $\beta_{m, V} \in C^\infty(\T^2; \R)$ to be determined and write the Fourier expansion of $T$: 
    \begin{equation}\label{T}
        T(t, x) =  \sum_{k, \ell \in \Z} ( v_{k, \ell} +i (k m -\ell) \beta_{{m, V}_{k, \ell}})e^{ikx+i\ell t},
    \end{equation}
where $V(t, x)= \sum_{k, \ell} v_{k, \ell} e^{i(kx+\ell t)}$ and $\beta_{m, V}(t, x)= \sum_{k, \ell \in \Z} \beta_{{m, V}_{k, \ell}} e^{i(kx+\ell t)}$. Define 
\begin{equation}\label{bkl}
\beta_{{m, V}_{k, \ell}}:=
\begin{cases}
    \frac{ v_{k, \ell}}{i(\ell-k m)} & \mbox{ if }  \ell  \neq k m , \\
    0 & \mbox{ otherwise.} 
\end{cases}
\end{equation}
The corresponding function $\beta_{m, V}(t,x):= \sum \beta_{{m, V}_{k, \ell}} e^{i(kx+ \ell t)}$ is a well defined smooth function over $\T^2$. This can be directly checked using that $V$ is smooth and that, since $m \in \N^+$, the sequence $\{\ell-km\}_{\ell \neq km}$ in \eqref{bkl} never accumulates to zero. Substituting the Fourier coefficients of $\beta_{m, V}$ (see \eqref{bkl}), in the expression of $T(t, x)$ in \eqref{T}, one obtains \eqref{find.beta}, concluding the proof.
\end{proof}

\section{Long-time Stability}\label{sec.stability}
In this section we prove the first point of Theorem \eqref{main.thm}. We anticipate here the idea of the proof: first we take any $V \in C^\infty(\T^2; \R)$ and, using the diffeomorphisms of the torus introduced in Section \ref{sec.conj}, we conjugate equation \eqref{initial.transp.eq} to a new transport equation having as leading transport term the resonant average $\la V \ra_m$ of $V$ (see \eqref{const.order.N} and recall Definition \ref{resonant.average} of resonant average). This is done in Lemma \ref{lemma.normal}.

Next, we exploit the characterization of the set $\cV_{m, s}$ to reduce such equation to constant coefficients (up to arbitrarily small in size remainders) and obtain the estimate in \eqref{stable.norms}. Precisely, the main result of this section is the following Proposition: 
\begin{proposition}\label{prop.const.coeff}
    For any $V \in \cV_{m, s}$ in \eqref{Vs}, for any $N \in \N^+ $ there exists $\e_N(V)>0$ such that, for any $0< \e < \e_N(V)$, there exists a $2\pi$-periodic in time, invertible linear map $\Phi_{N, \e}(t)$, $\hat{m}_N \in \R \setminus \{0\}$ and a smooth function $W_N \in C^\infty(\T^2; \R)$ such that 
    \begin{enumerate}
        \item for all $s \in\R$, there exists $C_{N, s} >0$ such that, for all $0\leq \e \leq\e_N(V)$ and for all $t \in \R$
 \begin{equation}\label{bound.Phi}
\left\|\Phi_{N, \e}(t)\right\|_{\cH^s \to\cH^s}+\left\|\Phi_{N, \e}(t)^{-1}\right\|_{\cH^s \to\cH^s }\leq C_{N, s}; 
 \end{equation}

 \item for any solution $u$ of equation \eqref{initial.transp.eq}, defining $v_N(t,x)$ as $u= \Phi_{N, \e}(t)v_N$ , then $v_N$ solves 
    \begin{equation}\label{const.order.N}
        \pa_t v_N = \Opw(i\xi \e (\hat{m}_N + \e^{N} W_N(t,x))) v_N, \quad \forall t \in \R. 
    \end{equation}
    \end{enumerate}
 \end{proposition}

 In the rest of this Section we will prove Proposition \ref{prop.const.coeff}, but we first use it to conclude the proof of point 1 of Theorem \ref{main.thm}.

\begin{proof}[Proof of Theorem \ref{main.thm}, point 1.]
    Fix $V\in \cV_{m, s}$, $N>0$, let $\e_N(V)>0$ be the threshold given by Proposition \ref{prop.const.coeff} and, for every $0< \e < \e_N(V)$ let $\Phi_{N, \e}$ be the corresponding transformation given by Proposition \ref{prop.const.coeff}. Recall that, defining $v_N(t,x)$ as $u= \Phi_{N,\e}(t) v_N$,  then $v_N$ solves \eqref{const.order.N}. Since $\Phi_{N, \e}(t)$ satisfies condition \eqref{bound.Phi} it is enough to prove that, for every $s \in \R$ there exists $C_{V,s}>0$ such that, for any $v_{N_0} \in C^\infty(\T)$, 
    \begin{equation}\label{stable.v}
        \norm{v_N(t)}_s \leq 2 \norm{v_{N_0}}_s \quad \forall \, |t| \leq C_{V,s} \e^{-(N+1)},
    \end{equation}
    where $v_N(t,x)$ is the solution of \eqref{const.order.N} satisfying $v_N(0,x)=v_{N_0}(x)$. Indeed, from the definition of $v_N$, this gives \eqref{stable.norms} with $C_1: = 2 C_{\Phi_N, s}^2$ (recall \eqref{bound.Phi}). In order to prove \eqref{stable.v}, we claim that there exists $C_{W_N, s}>0$ such that
    \begin{equation}\label{claim.stable}
        \frac{\di}{\di t}\norm{v_N}_s^2 \leq 2 \e^{N+1}C_{W_N, s} \norm{v_N}_s^2, \quad  \forall t, s\in \R.
    \end{equation}
    This immediately implies \eqref{stable.v} since integrating \eqref{claim.stable} one obtains 
    \begin{equation}
        \norm{v_N}_s \leq e^{C_{W_N, s} \e^{N+1} t} \norm{v_{N_0}}_{s}, \quad \forall \ t, s \in \R, 
    \end{equation}
    which is \eqref{stable.v} with $C_{V,s}:= \log(2) C_{W_N,s}^{-1}>0$. Thus we are left to show \eqref{claim.stable}. To this aim, we introduce the compact notation $H_N(t):=\Opw(i\xi\e(\hat{m}_N + \e^{N} W_{N}(t,x)))$, so that equation \eqref{const.order.N} reads $\pa_t v_N= H_N(t) v_N$. We obtain via a direct calculation, using skewadjointness of $H_N$
    
    \begin{equation}\label{derive.norm2}
        \frac{\di}{\di t} \norm{v_N}_s^2 = 2 \text{Re}( \la [ \Opw( \la \xi \ra^s) , H_N] v_N, \Opw( \la \xi \ra^s) v_N \ra).
    \end{equation}
    Next, using the explicit expression of $H_N(t)$ we have
    \begin{equation}\label{derive.norm3}
        [ \Opw( \la \xi \ra^s) , H_N] = \e^{N+1}[\Opw(\la \xi \ra^s), \Opw(i\xi W_N(t,x))] \in C^\infty(\T; \cS^{s}).
    \end{equation}
    Thus, via Cauchy-Schwarz inequality we get
    \begin{equation}
        \frac{\di}{\di t} \norm{v_N}_s^2 \stackrel{\substack{\eqref{derive.norm2}\\ \eqref{derive.norm3}}}{\leq} 2 \e^{N+1} \norm{ [ \Opw(\la \xi \ra^s), \Opw(i\xi W_N(t,x))]v_N}_{L^2} \norm{v_N}_{s} \leq 2 \e^{N+1}C_{W_{N, s}} \norm{v_N}_{s}^2,
    \end{equation}
    with $C_{W_N,s}:= \norm{[\Opw(\la \xi \ra^s), \Opw(i\xi W_N(t,x))]}_{\cH^s \to L^2}$, proving inequality \eqref{claim.stable} and concluding the proof. 
    \end{proof}

We now turn to the proof of Proposition \ref{prop.const.coeff}. We first prove the following Lemma, to reduce the transport operator of equation \eqref{initial.transp.eq} to a new one, having as leading term the resonant average of the potential. Precisely, we have the following result: 

\begin{lemma}\label{lemma.normal}
Let $V \in C^\infty(\T^2; \R)$ and take $u(t,x)$ any solution of equation \eqref{initial.transp.eq}. For every $N \in \N^+$ there exists $\e_N(V)>0$ such that, for every $0< \e < \e_N(V)$ there exists a $2\pi$-periodic, linear, invertible transformation $\Psi_{N, \e}(t): \cH^s \to \cH^s$, bounded for every $s \in \R$ and with bounded inverse, such that, defining $u_N(t, x)$ as $u(t,x)= \Psi_{N, \e}(t) u_N(t,x)$, then $u_N(t,x)$ solves
 \begin{equation}\label{vorrei}
     \pa_t u_N= \Opw(i\xi(\e \la V \ra_m(x) + \e^2  Z_N(x) + \e^{N+1} \tilde{W}_N(t,x)))u_N, 
 \end{equation}
 where $Z_N \equiv Z_N(x, \e)  \in C^\infty(\T; \R)$, $\tilde{W}_N \equiv W_N(x, \e) \in C^\infty(\T^2; \R)$ have all seminorms uniformly bounded in terms of $\e \in [0,\e_N(V)]$.
\end{lemma}

\begin{remark}\label{always.ok}
    We remark that this first step can be performed for a generic $V \in C^\infty(\T^2; \R)$ and we still do not use the characterization of $\cV_{m, s}$ given in \eqref{Vs}. 
\end{remark}

\begin{proof}
Take $V \in C^\infty(\T^2; \R)$, $u(t,x)$ solution of \eqref{initial.transp.eq} and fix any $N \in \N^+$. We claim that there exists $\e_N(V)>0$ such that, for every $0< \e < \e_N(V)$ there exists a sequence of $2\pi$-periodic, linear and invertible transformations $\tilde{\Phi}_1(t), \ldots, \tilde{\Phi}_N(t)$ such that, defining $u_N^{(1)}$ as $u(t,x)= \tilde{\Phi}_1(t) \circ \ldots \circ \tilde{\Phi}_N(t) u_N^{(1)}(t,x)$, then $u_N^{(1)}(t,x)$ solves 
    \begin{equation}\label{eq.N}
    \begin{split}
        &\pa_t u_N^{(1)}=H_N(t) u_N^{(1)}, \\
        H_N(t):=\Opw(i\xi(m + \e &\la V\ra_m(x +mt) + \e^2 \tilde{Z}_{N}(t,x) + \e^{N+1} \hat{W}_N(t,x))),
    \end{split}
    \end{equation}
    where $\tilde{Z}_{N}(t,x)$ is real valued and completely resonant with respect to $m$ (see Definition \ref{resonant.average}), $\hat{W}_N \in C^\infty(\T^2; \R)$. 
    We stress that also in this case $\tilde{\Phi}_i(t) \equiv \tilde{\Phi}_{i, \e}(t)$, for all $i=1, \ldots, N$, but we omit the dependence on $\e$ to simplify the notation. Let us postpone the proof of \eqref{eq.N} and show first how to conclude the proof of the proposition: denote by $\cT_m$  the following  translation in space
    \begin{equation}\label{cU}
        [\cT_m u](t,x):= u(t, x -mt) , \quad \forall u \in C^\infty(\T^2), 
    \end{equation}
    and define the function $u_N$ as $u_N:= \cT_m  u_N^{(1)}$, where $u_N^{(1)}$ solves \eqref{eq.N}. From \eqref{eq.N} we see that $u_N$ solves
    \begin{equation}\label{quasi}
        \pa_t u_N = \Opw(i\xi(\e \la V \ra_m (x) + \e^2 \tilde{Z}_{N}(t, x-mt) + \e^{N+1} \hat{W}_N(t,x-mt))) u_N.
    \end{equation}
      Remark that, since $\tilde{Z}_{N}$ is completely resonant with respect to $m$ (see Definition \ref{resonant.average}), then $\tilde{Z}_{N}(t,x-m t)= \tilde{Z}_{N}(0, x)$. Thus defining $Z_N(x):= \tilde{Z}_{N}(0,x)$, and $\tilde{W}_N(t,x):= \hat{W}_N(t, x-mt)$, we obtain that $u_N$ solves \eqref{vorrei}, concluding the proof. Remark that $u=\Psi_{N, \e} u_N$, with  \begin{equation}\label{psi.N}
          \Psi_{N, \e}(t):= \tilde{\Phi}_1(t) \circ \ldots \circ \tilde{\Phi}_N(t) \circ \cT_m^{-1}. 
      \end{equation}
      We now turn to the proof of \eqref{eq.N}. To this aim, we use Lemma \ref{preliminary.o1} and proceed inductively. First consider the case $N=1$: denote by  
    \begin{equation}\label{H.1}
        H(t):= \Opw(i\xi( m +\e V(t,x))), 
    \end{equation}
    (so that equation \eqref{initial.transp.eq} reads $\pa_t u = H(t) u$) and apply Lemma \eqref{preliminary.o1} with $n=1$, $f=m$ and $g=V$: let $\beta \in C^\infty(\T^2; \R)$ to be fixed later and consider the associated transformation $\tilde{\Phi}_1(t)$ defined in \eqref{Phi} and define $u_1^{(1)}(t,x)$ as $u= \tilde{\Phi}_1(t)u_1^{(1)}$. Recall from previous discussion that $u_1^{(1)}(t,x)$ satisfies $\pa_t u_1^{(1)}= \tilde{\Phi}_1(t)_*H(t) u_1^{(1)}$. Using Lemma \ref{preliminary.o1} and the expression of the pushforward in \eqref{pushforward}, we have
    \begin{equation}\label{push.H1}
        \tilde{\Phi}_1(t)_*H(t)=\Opw(i\xi( m + \e ( V(t,x) + m \pa_x \beta(t,x) - \pa_t \beta(t,x)) + \e^{2} \cR_{1, \e}(t,x))),
    \end{equation}
    where $\cR_{1, \e}=\cR_{1, \e}^1+ \cR_{1, \e}^2$ is the sum of the remainders given by \eqref{exp.one1} and \eqref{exp.one2}. We now apply Lemma \ref{preliminary.find.beta} to fix $\beta$: there exists $\beta_{m, V}^1 \in C^\infty(\T^2; \R)$ such that 
    \begin{equation}\label{ecco.b}
        V(t,x)+ m \pa_x \beta_{m, V}^1(t,x) - \pa_t \beta_{m, V}^1(t,x)= \la V \ra_m(x+mt),
    \end{equation}
where we recall the definition of resonant average in \eqref{1d.average}. Thus $u_1^{(1)}(t,x)$ solves \eqref{eq.N} with $N \leadsto 1$, $\tilde{Z}_{1}=0$ and $\hat{W}_1(t,x)=\cR_{1, \e}(t,x)$. This concludes the proof of \eqref{eq.N} when $N=1$. Let us remark that we fix $\tilde\e_1(V):= \e_{\beta_{m, V}^1}>0$, where $\e_{\beta_{m, V}^1}$ is given by Lemma \ref{preliminary.o1} and is in turn fixed by the choice of $\beta^1_{m, V}$ in \eqref{ecco.b}.
  
Finally, we take any $0<n<N$ and consider equation \eqref{eq.N} with $N \leadsto n $. Remark that $H_n(t)$ in \eqref{eq.N} is a particular instance of \eqref{H.n} with $ f= m +\e \la V \ra_m(x+mt) + \e^2 \tilde{Z}_{n}(t, x)$ and $g(t,x)= \e \hat{W}_n(t,x)$. Thus we can apply Lemma \ref{preliminary.o1}: define $u_{n+1}^{(1)}(t,x)$ as $u_{n}^{(1)}:= \tilde{\Phi}_{n+1}(t) u_{n+1}^{(1)}(t,x)$,  where $\tilde{\Phi}_{n+1}(t)$ is given in \eqref{Phi}, (with $\beta(t,x)$ to be determined). Again $u_{n+1}^{(1)}$ solves $\pa_t u_{n+1}^{(1)}= \tilde{\Phi}_{n+1}(t)_*H_n(t)u_{n+1}^{(1)}$. From \eqref{exp.one1} and \eqref{exp.one2} we have
\begin{equation}\label{name1}
\begin{split}
    \tilde{\Phi}_{n+1}(t)_*H_n(t)u_{n+1}^{(1)} 
    &= \Opw(i\xi ( m + \e \la V \ra_m(x+mt) + \e^2 \tilde{Z}_{n}(t, x) \\
    &+ \e^{n+1}(\hat{W}_n(t,x) + m \pa_x \beta(t,x) - \pa_t \beta(t,x)) \\
    &+ \e^{n+2} \cR_{{n+1}, \e}(t,x))), 
\end{split}
\end{equation}
where $\cR_{{n+1}, \e}=\cR_{{n+1}, \e}^1+ \cR_{{n+1}, \e}^2$ is the sum of the remainders given by \eqref{exp.one1} and \eqref{exp.one2}. 
 We can now conclude as in the previous step: applying Lemma \ref{preliminary.find.beta}, we can find $\beta_{m, V}^n(t,x) \in C^\infty(\T^2; \R)$ such that 
    \begin{equation}\label{name2}
        \hat{W}_n(t,x) + m \pa_x \beta_{m, V}^n(t,x) - \pa_t \beta_{m, V}^n(t,x)= \la \hat{W}_n \ra_m (x +mt),
    \end{equation}
    (recall \eqref{1d.average}). Substituting \eqref{name2} in \eqref{name1} we obtain that $u_{n+1}^{(1)}(t,x)$ solves \eqref{eq.N} with $N \leadsto n+1$, $\tilde{Z}_{n+1} := \tilde{Z}_{n} +\e^{n+1} \la \hat{W}_{n} \ra_m (x+ mt)$, which is still completely resonant with $m$ (see Definition \ref{resonant.average}) and $\hat{W}_{n+1}(t,x):= \cR_{{n+1}, \e}$. Finally define $\tilde{\e}_n(V):= \e_{\beta^n_{m,V}}>0$ given by Lemma \ref{preliminary.o1}, in turn fixed by $\beta_{m, V}^n$. 
    This concludes the proof of the Lemma, choosing $\e_N(V):=\min \{\tilde{\e}_1(V), \ldots, \tilde{\e}_N(V) \}>0$.
\end{proof}

We are now in position to prove Proposition \ref{prop.const.coeff}. 

\begin{proof}[Proof of Proposition \ref{prop.const.coeff}.]
Fix $V \in \cV_{m, s}$,  $N \in \N^+$ and let $\e_N(V)>0$ and $\Psi_{N, \e}(t)$  be respectively the threshold and the transformation given by Lemma \ref{lemma.normal}. Recall that, defining $u_N$ as $u(t,x) =[\Psi_{N, \e}(t)u_N](t,x)$, then $u_N$ solves equation \eqref{vorrei}.
Let now 
\begin{equation}\label{Lambda}
    [\Lambda u](t,x):= ( 1 +\lambda'(x))^\frac12 u(t, x + \lambda(x)), \quad \forall u \in C^\infty(\T^2)
\end{equation}
where $\lambda \in C^\infty(\T;\R)$ has to be determined. We claim that it is possible to choose $\lambda \in C^\infty(\T; \R)$ such that, considering the corresponding transformation $\Lambda$ in \eqref{Lambda} and defining $u_N=\Lambda v_N$, where $u_N$ solves \eqref{vorrei}, then $v_N$ solves \eqref{const.order.N}. Moreover such transformation $\Lambda$ is bounded and invertible with inverse bounded in $\cH^s$ for every $s \in \R$. This would conclude the proof with $\Phi_{N, \e}(t):=\Psi_{N, \e}(t) \circ \Lambda$.

To prove our claim, we first remark that if $1+ \lambda'(x) \neq 0$ for any $x \in \T$, then $\Lambda$ in \eqref{Lambda} is invertible with inverse of the form 
\begin{equation}\label{inv.Lambda}
    [\Lambda^{-1} u](t, y):= (1 + \tilde\lambda '(y))^\frac12 u(t, y + \tilde{\lambda}(y)), \quad \forall u \in C^\infty(\T^2), \ (t, x) \in \T \times \T,
\end{equation}
for some $\tilde{\lambda} \in C^\infty(\T;\R)$. 
We now proceed to identify the claimed function $\lambda \in C^\infty(\T; \R)$ assuming that $1 + \lambda'(x) \neq 0$ and thus formally writing $\Lambda^{-1}$ as in \eqref{inv.Lambda}. We will check in the end that the condition $1 + \lambda'(x) \neq 0$ is satisfied.
Denote by 
\begin{equation}\label{omega}
\omega(t,x):= \e \la V \ra_m (x) + \e^2 Z_N(x) + \e^{N+1} \tilde{W}_N(t,x),
\end{equation}
where $Z_N$ and $\tilde{W}_N$ are given by Lemma \ref{lemma.normal}, so that equation \eqref{vorrei} reads $\pa_t u_N= \Opw(i\xi \omega(t,x)) u_N$.  Then from \eqref{pushforward} $v_N$ defined as $u_N= \Lambda v_N$ solves 
\begin{equation}
    \pa_t v_N= \Lambda^{-1} \Opw(i\xi \omega(t,x)) \Lambda v_N
\end{equation}
(recall \eqref{omega} and \eqref{Lambda} and notice that the transformation $\Lambda$ is time independent). Noticing that 
\begin{equation}
    \Opw(i\xi \omega(t,x))=\omega(t,x) \pa_x +\frac12 \pa_x(\omega(t,x))
\end{equation}
and performing a direct computation one can show that 
\begin{equation}\label{push.lambda}
\begin{split}
    \Lambda^{-1} \Opw(i\xi \omega(t,x)) \Lambda & = \omega(t, y)( 1+ \lambda'(y)) \vert_{y=x + \tilde{\lambda}(x)} \pa_x  + \frac12 \pa_x \left(\omega(t, y)( 1+ \lambda'(y))\vert_{y=x + \tilde{\lambda}(x)} \right)\\
    &\stackrel{\eqref{comp.symbols}}{=} \Opw( i\xi \omega(t, y)( 1+ \lambda'(y))\vert_{y=x + \tilde{\lambda}(x)}). 
\end{split}
\end{equation}
Recalling now the explicit expression of $\omega(t,x)$ in \eqref{omega}, we can rewrite 
\begin{equation}\label{push.lambda2}
\begin{split}
    \Lambda^{-1} \Opw(i\xi \omega(t,x)) \Lambda & \stackrel{\eqref{push.lambda}}{=} \Opw\left(i\xi( \e \la V \ra_m (y) + \e^2 Z_N(y))(1+\lambda'(y)) \vert_{y=x + \tilde{\lambda}(x)} \right) \\
    & + \e^{N+1}\Opw\left(\tilde{W}_N(t, y)(1+ \lambda'(y))\vert_{y=x +\tilde{\lambda}(x)} \right).
\end{split}
\end{equation}
We now find $\lambda \in C^\infty(\T)$ such that 
\begin{equation}\label{cond.const}
    ( 1 + \lambda'(x))( \la V \ra_m(x) + \e Z_N(x))= \hat{m}_N,
\end{equation}
for some $\hat{m}_N \in \R$. Notice  that, since $V \in \cV_{m, s}$, up to shrinking $\e_N(V)>0$ we can suppose that, for every $0< \e < \e_N(V)$, the function $\la V \ra_m (x) + \e Z_N(x)$ does not have any zero (recall \eqref{Vs}). Thus \eqref{cond.const} immediately implies $1 + \lambda'(x) \neq 0$ proving invertibility of $\Lambda$ in \eqref{Lambda}. Plugging \eqref{cond.const} in \eqref{push.lambda2} we obtain \eqref{const.order.N} with $W_N(t,x):= \tilde{W}_N(t, y)(1+ \lambda'(y))\vert_{y=x +\tilde{\lambda}(x)}$ concluding the proof. 

To find $\lambda$ satisfying \eqref{cond.const}, we first rewrite it as 
\begin{equation}\label{constant2}
    \lambda'(x) = \left(\frac{\hat{m}_N}{\la V \ra_m(x) +\e Z_N(x)}-1\right),
\end{equation}
using again that $\la V \ra_m + \e Z_N(x)$ has no zeroes.
Next, requiring $\lambda \in C^\infty(\T)$ and integrating equation \eqref{constant2} on $\T$ one finds 
\begin{equation}
    \hat{m}_N = 2 \pi\left( \int_0^{2\pi} (\la V \ra_m(x) + \e Z_N(x))^{-1}\di x \right)^{-1} \in \R.
\end{equation}
Remark that, since $V \in \cV_{m, s}$ and since $\la V \ra_m + \e Z_N(x)$ has no zeroes, then the function $(\la V \ra_m(x) + \e Z_N(x))^{-1}$ has constant sign, and thus $\hat{m}_N$ can not be zero. 

Finally, in order to find $\lambda(x)$, write its Fourier series as $\lambda(x)=\sum_{k \in \Z} \lambda_k e^{ikx}$, for all $x \in \T$. 
Next, let $f(x):= \left(\frac{\hat{m}_N}{\la V \ra_m(x) +\e Z_N(x)}-1\right)$ and write its Fourier series as $f(x)=\sum_{k \in \Z} f_{k} e^{ikx}$. Remark that, from \eqref{constant2}, $f(x)$ has zero average, i.e., $f_{0}=0$. 
 Thus \eqref{constant2} reads 
\begin{equation}\label{lambda.k}
    \lambda_{k} = \frac{f_{k}}{ik} \quad \forall k \neq 0, \quad \lambda_0=0.
\end{equation}
One can check that, since $f \in C^\infty(\T; \R)$, then $\lambda(x):=\sum_{k \in \Z} \lambda_{k} e^{ikx}$, with $\{\lambda_{k}\}_{k \in \Z}$ given by \eqref{lambda.k} is a well defined, real valued and smooth function as well. By construction $\lambda$ verifies \eqref{cond.const}, concluding the proof. 
\end{proof}

\section{Instability}\label{sec.instability}
In this section we prove the second point of Theorem \ref{main.thm}. As anticipated, the core of the proof is the construction of an escape function for $h(x, \xi)= \xi \la V \ra_m(x)$, with $V \in \cV_{m, u}$ (see \eqref{Vu}). This is done in Section \ref{sec.escape}, adapting the construction of \cite{Colin_de_Verdi_re_2020} (see also \cite{LangellaMasperoRotolo25}). Next, in section \ref{sec.energy.est} we prove existence of unstable solutions performing first a normal form reduction as the one in Lemma \ref{lemma.normal} and then a positive commutator estimate in which we exploit the properties of the escape function. 

\subsection{Construction of the escape function}\label{sec.escape}
In this section we consider a smooth function 
\begin{equation}\label{h}
    h(x, \xi):= \xi X(x), \quad (x, \xi) \in T^*\T \simeq \T \times \R, \ X \in C^\infty(\T), 
\end{equation}
we assume that that $X$ is non degenerate and has at least one zero (see Definition \ref{non.deg}) and we construct an \emph{escape function} for $h$. We denote by $\cX_h$ the Hamiltonian vector field of $h$ on $T^*\T$, that in coordinates reads 
\begin{equation}\label{Xh}
    \mathcal{X}_h(x, \xi)= ( X(x), - \xi \pa_x X(x)), \quad \forall (x, \xi) \in T^*\T,
\end{equation}
and we denote by $\Phi^t_{\cX_h}(x,\xi)$ the flow of $\cX_h$ on $T^*\T$. 
Finally, for every $f \in C^\infty(T^*\T)$ we write the Poisson brackets between  $h$ and $f$ as  
\begin{equation}\label{poisson.bracket}
    \{ h , f \} := \pa_\xi h \pa_x f - \pa_x h \pa_\xi f  .
\end{equation}    

We now give the definition of escape function. For all this section we will use the notation $T^* \T \setminus \{ 0 \} = \{ (x, \xi) \in T^*\T : \xi \neq 0 \}$.  
\begin{definition}\label{def:homo}
We say that a smooth function $f: T^*\T \setminus \{ 0 \} \to \R$ is positively homogeneous of degree $\rho$ if 
$$
f(x, \lambda \xi)= \lambda^{\rho}f(x, \xi), \quad \forall (x, \xi) \neq 0 ,  \ \forall \ \lambda > 0.
$$
\end{definition}

\begin{definition}[Escape function]\label{def:escape}
    Let $h: T^*\T \to \R$ be as in \eqref{h}. We say that a function $a \in C^\infty(T^*\T \setminus \{ 0 \})$, positively homogeneous of degree one (see Definition \ref{def:homo}) 
    is an {\em escape function} for $h$ if there exists some positive $\delta$ such that 
    \begin{equation}\label{Poiss.esc}
    \{h, a\} \geq \delta |\xi|, \quad \forall (x, \xi) \in T^*\T \setminus \{ 0 \}. 
    \end{equation}
\end{definition}
 The main results of this section is the following proposition. 
 
\begin{proposition}\label{escape1}
    Let $X \in C^\infty(\T)$ and suppose that $X$ has at least one zero and that all the zeroes are non degenerate (see Definition \ref{non.deg}). Then the function $h$ defined in \eqref{h} admits an escape function in the sense of Definition \ref{def:escape}. Moreover, there exists an open subset $\cW \subset \T$ such that 
    \begin{equation}\label{pos.a}
        a(x, \xi) \leq - \frac{|\xi|}{2}, \quad \forall x \in \cW.
    \end{equation}
\end{proposition}

\paragraph{Dynamical properties.}
In order to prove Proposition \ref{escape1} we need some dynamical properties of the Hamiltonian flow induced by $h(x, \xi)$ in \eqref{h}, that we list in this section. First of all, we recall (see, for example, \cite[Chapter 9]{lee2003introduction}) that a smooth vector field is \emph{complete} if it generates a global flow, or equivalently if, for every initial datum, the associated flow exists for all in times. We have the following result, for whose proof we refer to \cite[Section 5.2]{Taira}. 

\begin{lemma}\label{global.ex}
    Let $h$ as in \eqref{h}. Then the Hamiltonian vector field $\cX_h$ (see \eqref{Xh}) is complete.
\end{lemma}

\begin{remark}\label{homo.flow}
    Denoting by $\Phi^t_{\cX_h}(x,\xi)=(X^t(x, \xi), \Xi^t(x, \xi))$ the components of the flow of $\cX_h$ in \eqref{Xh} with initial datum $(x, \xi) \in T^*\T$, the following homogeneity property holds:
    \begin{equation}\label{eq.homo.flow}
    (X^t(x, \lambda \xi), \Xi^t(x, \lambda \xi))= (X^t(x, \xi), \lambda \Xi^t(x, \xi)), \quad \forall \lambda \in \R,  \  \forall (x, \xi) \in T^*\T, \ \forall t \geq 0. 
    \end{equation}
\end{remark}

We now identify global attractors and repellors for the flow. In particular, we show that the dynamics of $\cX_h$ on $T^*\T$ is completely determined once the zeroes of $ X $  are known. To this aim let us first introduce the following terminology: 

\begin{definition}[Attractor-Repellor structure]\label{def:attractor}
We say that the flow $\Phi^t_{\cX_h}$ on $T^*\T$ has a \emph{global attractor-repellor} structure if there exist $A$, $R$ closed invariant subsets of $T^*\T$ such that:
\begin{equation}
    \begin{split}
        &\mbox{ for every } (x, \xi) \in T^*\T \setminus R, \quad \Phi^t_{\cX_h}(x, \xi) \to A \quad \mbox{as}  \quad t \to + \infty,\\
        &\mbox{ for every } (x, \xi) \in T^*\T \setminus A \quad \Phi^t_{\cX_h}(x, \xi) \to R \quad \mbox{as}  \quad t \to - \infty.\\ 
    \end{split}
\end{equation}
We denote the basins of attraction of $A$ and $R$ respectively as $B(A):=T^*\T \setminus (R)$ and $B(R):=T^*\T \setminus (A)$.
\end{definition}

\begin{lemma}\label{structure.flow}
    Let $h(x, \xi)=  \xi X(x)$ and suppose that $X$ has at least one zero, but does not have any degenerate zero (see Definition \ref{non.deg}).  Denote by 
    \begin{equation}\label{setK}
    K^+:=\{ x \in \T: X(x)=0,\ X'(x)< 0 \} \quad \mbox{ and} \quad K^-:=\{ x \in \T:X(x)=0,\  X'(x) >0 \}, 
\end{equation}
and by 
\begin{equation}\label{setGamma}
    \Gamma^+:=\{(x, \xi) : x \in K^+\} \simeq K^+ \times \R  \quad \mbox{ and } \quad \Gamma^-:=\{(x, \xi) : x \in K^-\} \simeq K^- \times \R.
\end{equation}
Then the sets $\Gamma^+$, $\Gamma^-$ and $\cT:= \{\xi=0\}$ are invariant for the flow of $\cX_h$ in \eqref{Xh}. Moreover the flow $\Phi^t_{\cX_h}$ has a global attractor-repellor structure (see Definition \ref{def:attractor}) with attractor $\Gamma^+$ and repellor $\Gamma^-$. In particular, for all $(x_0, \xi_0) \in T^*\T \setminus ( \Gamma^+ \cup \Gamma^-)$ we have
\begin{equation}\label{att.rep}
\Phi^t_{\cX_h}(x_0, \xi_0) \to \Gamma^+, \ \mbox{ as } t \to +\infty, \quad \mbox{ and } \quad \Phi^t_{\cX_h}(x_0, \xi_0) \to \Gamma^-, \ \mbox{ as } t \to -\infty. 
\end{equation}
\end{lemma}

\begin{proof}
First of all remark that, since $X$ has no degenerate zeroes, then
\begin{equation}\label{all.zeroes}
    X^{-1}(0)=K^+ \cup K^-.
\end{equation}
Moreover, using the explicit expression of $\cX_h$ in \eqref{Xh}, one can check that $\cT$, $\Gamma^+$ and $\Gamma^-$ are the only invariant sets for the flow. We now prove \eqref{att.rep}: let $(x_0, \xi_0) \in T^*\T \setminus (\Gamma^+ \cup \Gamma^-)$ and recall that, from Lemma \ref{global.ex}, the flow $\Phi^t_{\cX_h}(x_0, \xi_0)$ exists for all $t \in \R$. Thus, in order to prove \eqref{att.rep} we have to show that $x(t) \to K^+$ as $t \to +\infty$ and $x(t) \to K^-$ as $t \to - \infty$ (see \eqref{setK} and \eqref{setGamma}). This follows immediately from the fact that the dynamics of $x(t)$ is determined by the one dimensional system $\dot{x}=X(x)$ (see \eqref{Xh}) and from the definitions of $K^+$ and $K^-$ in \eqref{setK}, thus we omit the details of this last part of the proof.
\end{proof}

\begin{remark}\label{rmk.inv}
    From the structure of the flow described in Lemma \ref{structure.flow} one can check that any conical neighborhood $\cU^+$ of $\Gamma^+$ is positively invariant for the flow (i.e., $\Phi^t_{\cX_h}(\cU^+) \subset \cU^+$, for all $t>0$) and analogously any conical neighborhood of $\Gamma^-$ is invariant for the flow at negative times. 
\end{remark}

\subsubsection{Construction of the escape function}
In this section we use the structure of the flow of $\cX_h$ (see \eqref{Xh}) described in Lemma \ref{structure.flow} and the sets defined therein (\eqref{setK},\eqref{setGamma}) to construct an escape function for $h(x, \xi)$ in the sense of Definition \ref{def:escape}. We first construct a local escape function close to the zeroes of $X$ and then we progressively extend it to the whole space. 

\begin{lemma}\label{lemma.local}
    Let $h(x, \xi)=\xi X(x)$ and suppose that $X \in C^\infty(\T)$ has at least one zero and is non degenerate (see Definition \ref{non.deg}). There exist conical neighborhoods $\cU^\pm$ of $\Gamma^\pm$ (see \eqref{setGamma}), $\delta>0$ and a smooth function $k\in C^\infty(T^*\T \setminus \{ 0\})$ such that 
    \begin{equation}\label{kappa}
     \quad k(x, \xi)= \pm |\xi| \ \mbox{ on } \Gamma^\pm, \quad \mbox{ and } \quad \{h, k\} \geq \delta |\xi|, \quad \forall (x, \xi) \in \cU^+ \cup \cU^-.  
    \end{equation}
    Moreover the function $k$ is positively homogeneous of degree one (see Definition \ref{def:homo}) on its domain.
\end{lemma}
\begin{proof}
    Denote by $\{x_1, \ldots, x_n\}$ the zeroes of $X(x)$ and, using non-degeneracy of $X$, let $\nu:= \min_i  |X'(x_i)| >0$. Let $\tilde{\cU}^\pm \subseteq \T$  be neighborhoods of $K^\pm$ in $\T$ (see \eqref{setK}) to be fixed later, and let $\tilde{k} \in C^\infty(\T)$ be such that 
    \begin{equation}\label{localk}
         \tilde{k}(x)=
    \begin{cases}
        1 & \mbox{ if } x \in \tilde{\cU}^+\\
        -1 & \mbox{ if } x \in \tilde{\cU}^-.
    \end{cases}
    \end{equation}
    Set $k(x, \xi):= |\xi| \tilde k(x)$. Clearly $k$ is positively homogeneous of degree one and $k=\pm|\xi|$ on $\Gamma^\pm$. Moreover, computing $\{ h, k \} (x_i)$ for any $i \in \{1, \ldots, n \}$ we also have
    $$
    \{ h, k\}(x_i, \xi)= -|\xi| X'(x_i) \tilde{k}(x_i) = |\xi| |X'(x_i)| \geq |\xi| \nu, \quad \forall i=1, \ldots, n,
    $$
    where for the last equality we have used \eqref{localk} and the definition of $K^\pm$ in \eqref{setK}. 
    Thus, up to choosing $\tilde{\cU}^\pm$ small enough, we have by continuity 
    $
    \{h, k \} \geq \frac{\nu}{2}|\xi| \mbox{ on } \cU^+ \cup \cU^-
    $
    (where $\cU^\pm := \tilde{\cU}^\pm \times \R$), concluding the proof of \eqref{kappa} with $\delta:=\frac{\nu}{2}$. 
\end{proof}

The next goal is to extend the local escape function $k$ to the basins of attractions $B(\Gamma^\pm):=T^*\T \setminus \Gamma^\mp$ (see Definition \ref{def:attractor}). To this aim we first need the following auxiliary function, which extends the Poisson brackets $\{h, k\}$, defined on $\cU^\pm$, to a positive function defined on the whole $T^*\T$.

\begin{lemma}\label{lemma.m}
    There exists a smooth function $\tm:  T^*\T \setminus \{ 0 \} \to \R$, positively homogeneous of degree one, such that 
    \begin{equation}\label{fun.m}
        \tm(x, \xi)=\{h, k\}(x, \xi), \ \forall (x, \xi) \in \cU^\pm, \quad \mbox{ and } \quad \tm(x, \xi) \geq \frac{\delta}{2}|\xi|, \quad \forall \ (x, \xi) \in T^*\T,
    \end{equation}
    where $k$, $\delta$ and $\cU^\pm$ are as in Lemma \ref{lemma.local}.
\end{lemma}

\begin{proof}
    Let $\cU^\pm$ be the conical neighborhoods of $\Gamma^\pm$  given by Lemma \ref{lemma.local}. Denote by $\tilde{\cU}^\pm \subset \T$ the neighborhoods of $K^\pm$ obtained projecting $\cU^\pm$, i.e. 
    $
    \tilde{\cU}^\pm:=\{ x \in \T : \ (x, \xi) \in \cU^\pm, \ \forall \xi \in \R \}. 
    $
    Then it is possible to define $\tilde{\tm}: \T \to \R$, a smooth function satisfying 
    \begin{equation}\label{t.m.positive}
    \tilde{\tm}(x)=\{h, k\}(x, 1) \quad \forall \ x \in (\overline{\cU^+}\cup \overline{\cU^-}) \quad \mbox{ and } \tilde{\tm} \geq \frac{\delta}{2} \quad \forall \ x \in \T.
\end{equation}
    This can be done since, from \eqref{kappa}, it possible to extend $\{h, k \}(x, 1)$ in a smooth way outside $\overline{\tilde{\cU}^\pm}$, preserving positivity as in \eqref{t.m.positive}. 
Finally, we define $\tm : T^*\T  \setminus \{0\} \to \R$ extending $\tilde{\tm}$ by homogeneity, i.e. $m(x, \xi):=\tilde{m}(x)|\xi|$, for all $(x, \xi) \in T^*\T$, obtaining that $\tm$ is a positively homogeneous function of degree one and satisfies \eqref{fun.m}. 
\end{proof}

We can now extend the local escape functions to the sets $B(\Gamma^\pm):=T^*\T \setminus \Gamma^\mp$ (see Definition \ref{def:attractor}). In the next Lemma we write $z:=(x, \xi) \in T^*\T$ for simplicity in the notation.

\begin{lemma}\label{escape.basins}
    Let $z=(x, \xi) \in B(\Gamma^\pm)$ and define 
\begin{equation}\label{def:ell+}
    \ell^\pm(z):=\lim_{t \to \pm \infty}\Big[k(\Phi_{\cX_h}^t(z)) - \int_0^t \tm(\Phi_{\cX_h}^{s}(z))\di s \Big],
\end{equation}
where $k$ is given by Lemma \ref{lemma.local} and $\tm$ by Lemma \ref{lemma.m}. 
Then
\begin{itemize}
     \item[(i)] $\ell^\pm$ are well defined on $B(\Gamma^\pm)$, smooth, positively homogeneous of degree one and $\ell^\pm=k$ on $\Gamma^\pm$; 

    \item[(ii)] $\ell^\pm$ are escape functions on $B(\Gamma^\pm)$. In particular, let $\delta>0$ be given by Lemma \ref{lemma.local}, then 
    \begin{equation}\label{ell.poiss}
    \{ h, \ell^\pm\} \geq \frac{\delta}{2}|\xi|, \quad \forall z \in B(\Gamma^\pm).
    \end{equation}
\end{itemize}
\end{lemma}

\begin{proof}
We prove the statement with +.\\
    Let $z \in B(\Gamma^+)$ and let $t_0(z)  \geq 0$ be such that $\Phi_{\cX_h}^{t_0(z)}(z) \in \cU^+$, with $\cU^+$ the neighborhood of $\Gamma^+$ defined in Lemma \ref{lemma.local}. By Definition \ref{def:attractor} of $B(\Gamma^+)$, $t_0(z)$ is finite. We show that for any  $t>t_0$, the argument of the limit in \eqref{def:ell+} is constant.
    Since the set $\cU^+$ is positively invariant for the flow of $\cX_h$ (see Remark \ref{rmk.inv}), we have
$\tm(\Phi_{\cX_h}^t(z))= \{ h, k\}(\Phi_{\cX_h}^t(z))$  for any  $ t\geq t_0(z)$ (recall \eqref{fun.m}).
Therefore, for $t\geq t_0(z)$ one has
\begin{align}
    \notag
        & k(\Phi_{\cX_h}^t(z))-\int_{0}^t \tm(\Phi_{\cX_h}^s(z))\di s
        \notag
        = k(\Phi_{\cX_h}^t(z))-\int_0^{t_0(z)} \tm(\Phi_{\cX_h}^s(z))\di s-\int_{t_0(z)}^t \{h, k\}(\Phi_{\cX_h}^s(z))\di s\\
        \notag
        = & k(\Phi_{\cX_h}^t(z))-\int_0^{t_0(z)} \tm(\Phi_{\cX_h}^s(z))\di s- k(\Phi_{\cX_h}^t(z))+ k(\Phi_{\cX_h}^{t_0(z)}(z))\\
        \label{ell2}
        = & -\int_0^{t_0(z)} \tm(\Phi_{\cX_h}^s(z))\di s + k(\Phi_{\cX_h}^{t_0(z)}(z)),
    \end{align}
    and thus the argument of the limit is definitely constant, proving that $\ell^+$ is well defined. Moreover, if $z \in \Gamma^+$, then $t_0(z)=0$ and \eqref{ell2} reads $\ell^+(z)=k(z)$.

\noindent \underline{Smoothness:} 
fix $z_0 \in B(\Gamma^+)$ and let $V_{z_0} \subset B(\Gamma^+)$ be a neighborhood of $z_0$. We observe that, from the global picture of the flow of $\cX_h$ in Lemma \ref{structure.flow} and Remark \ref{rmk.inv}, there exists $t_1(V_{z_0})>0,$ possibly larger than $t_0(z_0),$ such that $ \Phi^{t}_{\cX_h}(z) \in \cU^+$ for all $z \in V_{z_0}$ and $t \geq t_1(V_{z_0})$. 
Then,  by \eqref{ell2}, on the set $V_{z_0}$ the function $\ell^+$ reduces to
\begin{equation}\label{ell.piu.v}
\ell^+(z)=  -\int_0^{t_1(V_{z_0})} \tm(\Phi_{\cX_h}^s(z))\di s + k(\Phi_{\cX_h}^{t_1(V_{z_0})}(z)) , \quad \forall z \in V_{z_0}  \ ,
\end{equation}
which is smooth.

\noindent \underline{Homogeneity:}
Using remark \ref{homo.flow} and recalling the homogeneity properties of $\tm$ and $k$ given in Lemma \ref{lemma.m} and \ref{lemma.local} respectively, then one can directly check from the definition of $\ell^+$ in \eqref{def:ell+} that $\ell^+(x, \lambda \xi)= \lambda \ell^+(x, \xi)$, for all $\lambda>0$, proving that $\ell^+$ is positively homogeneous of degree one.  \\

\noindent $(ii)$
Given $z_0 \in B(\Gamma^+)$ and $V_{z_0} \subset B(\Gamma^+)$ a small neighborhood of $z_0$, let $t_1(V_{z_0})$ as above. We have 
\begin{equation*}
\begin{split}
        \{h, \ell^+\}(z) 
        &\stackrel{\eqref{ell.piu.v}}{=} \{h, k \circ \Phi_{\cX_h}^{t_1(V_{z_0})}\}(z)-
        \int_0^{t_1(V_{z_0})} \{ h, \tm \circ \Phi_{\cX_h}^s(z)) \} \di s\\
        &= \{h, k  \}(\Phi_{\cX_h}^{t_1(V_{z_0})}(z))
        -\frac{\di}{\di t}\Big|_{t=0}\int_0^{t_1(V_{z_0})} \tm(\Phi_{\cX_h}^{t+s}(z)) \di s  \\
        &= \{h, k  \}(\Phi_{\cX_h}^{t_1(V_{z_0})}(z)) -\tm(\Phi_{\cX_h}^{t_1(V_{z_0})}(z))+ \tm(z)\\
        &= \tm(z) \stackrel{\eqref{fun.m}}{\geq} \frac{\delta}{2} |\xi|,
\end{split}
\end{equation*}
concluding the proof. 
\end{proof}

Finally we conclude the construction gluing together the functions $\ell^+$ and $\ell^-$ of Lemma \ref{escape.basins}. 

\begin{lemma}\label{gluing}
    Given $\sigma>0$ sufficiently small, there exists a smooth function $\eta_\sigma:\T \to \R$, satisfying $\eta_\sigma(K^+)\equiv 1, \ \eta_\sigma(K^-)\equiv 0$, and such that 
    \begin{equation}\label{bound.eta}
            \left| \frac{\di}{\di t}\vert_{t=0} \eta_\sigma[X^t(x, \xi)]\right| < \sigma, \quad \forall (x, \xi) \in \T \times \R. 
    \end{equation}
    where we recall the notation in Remark \ref{homo.flow} for the components of  the flow of $\cX_h$. 
\end{lemma}

\begin{proof}
Recalling the definition of $K^-$ in \eqref{setK},
let $\tilde{\cU}^- \subseteq \T$ be an open neighborhood of $K^-$. Using again the dynamics of $X^t(x, \xi)$ (see \eqref{Xh}), we deduce that for any $x \notin K^+ \cup K^-$ there exists a unique $t_x \in \R$ such that $X^{-t_x}(x, \xi) \in \pa \tilde{\cU}^-$ (see Remark \ref{rmk.inv}).  Define then 
\begin{equation}\label{tilde.t}
    \tilde{t}(x):=
    \begin{cases}
        t_x, & x \in \T \setminus (K^+ \cup K^-) \\
        + \infty, & x \in K^+\\
        -\infty, & x \in K^-.\\
    \end{cases}
\end{equation}
Remark that for any $t \in \R$ one has 
\begin{equation}\label{property.t}
    \tilde{t}(X^t(x, \xi))=\tilde{t}(x)+t. 
\end{equation}

Next take a smooth non-decreasing function $\phi_\sigma: \R \to \R$, such that 
\begin{equation}\label{phi.e}
    \phi_\sigma(\tau):=
    \begin{cases}
        0, & \tau \leq -\sigma\\
        \sigma \tau, & \sigma < \tau \leq \sigma^{-1}-\sigma\\
        1, & \tau > \sigma^{-1}+\sigma, \\ 
    \end{cases}
    \quad \mbox{ and } \quad \sup_{\tau \in \R}|\phi'_\sigma(\tau)| \leq 10 \sigma. 
\end{equation}
We now prove that the function $\eta_\sigma$ claimed by the lemma is  $\eta_\sigma(x):=\phi_{\frac{\sigma}{10}}(\tilde{t}(x))$, for all $ x \in \T$, with $\tilde{t}$ defined in \eqref{tilde.t} and $\phi_{\frac{\sigma}{10}}$ in \eqref{phi.e}. 
First of all notice that, from the definition of $\tilde{t}$ in \eqref{tilde.t}, if $x \in K^+$ then
$
\eta_\sigma(x)=1
$ 
and analogously if $x \in K^-$, then $\eta_\sigma(x)=0$.
Next notice that, for any $(x, \xi) \in \T \times \R$ we have: 
$$
\left|\frac{\di}{\di t}\vert_{t=0} \eta_\sigma[X^t(x, \xi)]\right|=\left|\frac{\di}{\di t}\vert_{t=0} \phi_{\frac{\sigma}{10}}(\tilde{t}(X^t(x, \xi)))\right|\stackrel{\eqref{property.t}}{=}\left|\frac{\di}{\di t}|_{t=0}\phi_{\frac{\sigma}{10}}(\tilde{t}(x)+ t)\right|=\left|\phi'_{\frac{\sigma}{10}}(\tilde{t}(x))\right|\stackrel{\eqref{phi.e}}{<} \sigma.
$$
This proves point $(ii)$ and concludes the proof of the Lemma.
\end{proof}

We can now conclude the proof of Proposition \ref{escape1}.  
\begin{proof}[Proof of Proposition \ref{escape1}]
Define, for $\sigma >0$ sufficiently small and to be fixed later, 
    \begin{equation}
        a(x, \xi):=\eta_\sigma(x)\ell^+(x, \xi)+ (1-\eta_\sigma)\ell^-(x, \xi),
    \end{equation}
    with $\eta_\sigma$ defined in Lemma \ref{gluing} and $\ell^+$ and $\ell^-$ in Lemma \ref{escape.basins}. First of all notice that $a$ is well defined, since 
    $\eta_\sigma(x) =0$ for all $(x, \xi) \in \Gamma^-$, where $\ell^+$ is not defined and analogously $(1-\eta_\sigma)(x)=0$ for all $(x, \xi) \in \Gamma^+$ where $\ell^-$ is not defined.
   Next remark that $a$ is smooth in $T^*\T \setminus \{ 0 \}$, since both $\ell^\pm$ and $\eta_\sigma$ are and, from the homogeneity properties of $\eta_\sigma$ and $\ell^\pm$, we infer that $a$ is positively homogeneous of degree one. 
   
   Next we show property \eqref{pos.a}. We remark that $a(x, \xi) = -|\xi|$ on $\Gamma^-$: indeed, $\eta_\sigma=0$ on $K^-$ from Lemma \ref{gluing} and putting together \eqref{kappa} and point $(i)$ of Lemma \ref{escape.basins}, we get $\ell^-(x, \xi)= k(x, \xi)=-|\xi|$ on $\Gamma^-$. Thus from continuity of $a$ there exists $\cW$ neighborhood of $K^-$ such that \eqref{pos.a} holds for all $x \in \cW$.
   
   We are left to prove \eqref{Poiss.esc}. 
   We have
   \begin{equation}\label{h.a}
       \{h, a \} = \{ h, \eta_\sigma\}\ell^+ + \{h, \ell^+\}\eta_\sigma + \{h, \ell^-\}(1-\eta_\sigma)-\{h, \eta_\sigma\}\ell^- \stackrel{\eqref{ell.poiss}}{\geq} \frac{\delta}{2}|\xi| + \{h, \eta_\sigma\}(\ell^+-\ell^-).
   \end{equation}
    Finally, remark that
    $$
    \{h, \eta_\sigma\}(x, \xi)=X(x) \eta_{\sigma}'(x)= \frac{\di}{\di t}\vert_{t=0} \eta_{\sigma}(X^t(x, \xi)),
    $$
    thus using \eqref{bound.eta} we have $|\{h, \eta_\sigma\} | < \sigma$ and this gives
    \begin{equation}
        \{h, a\} \stackrel{\eqref{h.a}}{\geq} \frac{\delta}{2}|\xi| - |\{h, \eta_\sigma\}| |\xi| \underbrace{\max_{\T} \left| (\ell^+-\ell^-)\left(x, \frac{\xi}{|\xi|}\right)\right|}_{=:C} = |\xi|\left(\frac{\delta}{2} - C \sigma\right),
    \end{equation}
where for this last inequality we have used homogeneity of $\ell^\pm$. In conclusion, choosing $\sigma< \frac{\delta}{4C}$ we obtain \eqref{Poiss.esc} with $\delta \leadsto \frac{\delta}{4}$. 
\end{proof}

\subsection{Energy estimate}\label{sec.energy.est}

We are now in position to prove the second point of Theorem \ref{main.thm}. To do so, we need to first transform equation \eqref{initial.transp.eq} in a new transport equation, having $\la V \ra_m(x)$ as leading transport term (recall Definition \ref{resonant.average}). 
Precisely, given any $V \in \cV_{m, u}$, we apply Lemma \ref{lemma.normal} with $N=1$ (see Remark \ref{always.ok}): there exists $\e_1(V)>0$ and for every $0< \e < \e_1(V)$ a $2\pi$-periodic, linear transformation $\Psi_{1,\e}(t)$, such that, given $u(t,x)$ solution of equation \eqref{initial.transp.eq} and defining $u_1(t,x)$ as $u(t,x)=\Psi_{1,\e}(t) u_1(t,x)$, then $u_1(t,x)$ solves 
\begin{equation}\label{normal.instab}
    \pa_t u_1= \Opw(i \xi (\e \la V \ra_m(x) + \e^2 W(t,x)))u_1(t,x),
\end{equation}
for some $W \in C^\infty(\T^2;\R)$. Moreover, the transformation $\Psi_{1,\e}(t): \cH^s \to \cH^s$ is bounded for any $s \in \R$ with bounded inverse. We now show that there exists $u_{1_0} \in C^\infty(\T^2)$ such that, if $u_1(t,x)$ is the solution of \eqref{normal.instab} satisfying $u_1(0,x)=u_{1_0}(x)$, then for every $s>1$ we have
\begin{equation}\label{growth.u1}
    \norm{u_1(t)}_s \geq \tilde{\delta}_1 e^{ \e \delta_2 t} \norm{u_{1_0}}_s, \quad \forall t \in \R,
\end{equation}
for some $\tilde{\delta}_1, \delta_2>0$. Remark that, from boundedness of the transformation $\Psi_{1,\e}$ and of its inverse, this implies inequality \eqref{cond.growth}, concluding the proof. Thus in the rest of this section we prove \eqref{growth.u1}. 

\paragraph{Energy Estimate.}
In order to prove \eqref{growth.u1} first recall that, since $V \in \cV_{m, u}$, we can apply  Proposition \ref{escape1} to $h(x,\xi)=\xi \la V\ra_m(x)$: denote by $a(x, \xi) \in C^\infty(T^*\T \setminus \{ 0 \} )$ the escape function for $h(x,\xi)$ constructed in Proposition \ref{escape1}. We start by associating a symbol $\tilde{a} \in \bS^1(\T)$ (see \eqref{symbol}) to the escape function: take $\chi \in C_c^\infty(T^*\T ;[0, 1])$, satisfying $\chi(x, \xi) = 1$ for all $|\xi| \leq \frac12$ and $\chi(x, \xi)=0$ for all $|\xi| \geq 1$ and define 
\begin{equation}\label{tilde.a}
    \tilde{a}(x, \xi) = (1- \chi)(x, \xi) a(x, \xi), \quad \forall (x, \xi) \in \T \times \R. 
\end{equation}
Clearly $\tilde{a} \in C^\infty(T^*\T)$ and $\tilde{a}= a$ on $|\xi| \geq 1$. Using the homogeneity of $a$ one can check that $\tilde{a} \in \bS^1$. We now use a virial argument like in \cite{MasperoMurgante}:  given any $u_1(t,x)$ solution of \eqref{normal.instab}, we define 
\begin{equation}\label{At}
    \cA(t):= \la \Opw(-\tilde a) u_1(t), u_1(t) \ra_{L^2(\T)} \quad \forall t \in \R.
\end{equation}

We have the following result:
\begin{lemma}\label{lemma.At}
    Let $u_1(t,x)$ be a solution of equation \eqref{normal.instab} with initial condition $u_1(0,x)=u_{1_0}(x)$. There exist constants $\delta_2, \beta_{\e, W, \tilde{a}}>0$ such that 
\begin{equation}\label{ineq.At}
    \dot\cA(t) \geq \e \left( \delta_2 \cA(t) -\beta_{\e, W, \tilde{a}} \norm{u_{1_0}}^2_{L^2(\T)} \right), \quad \forall t \in \R,
\end{equation}
where $\cA(t)$ is defined in \eqref{At}. 
\end{lemma}

Before proving Lemma \ref{lemma.At} let us remark that, integrating inequality \eqref{ineq.At} one obtains 
\begin{equation}\label{A.1}
    \cA(t) \geq e^{\e \delta_2 t} \left(\cA(0) - \frac{\beta_{\e, W, \tilde{a}}}{\delta_2} \norm{u_{1_0}}_{L^2(\T)}^2 \right).
\end{equation}
On the other hand, using Cauchy-Schwarz inequality and property \eqref{e.cv}, we find $C_{\tilde{a}}>0$ such that
\begin{equation}\label{A.2}
    \cA(t) \stackrel{\eqref{At}}{\leq} \norm{ \Opw(-\tilde{a}) u_1(t)}_{L^2(\T)}\norm{u_1(t)}_{L^2(\T)} \stackrel{\eqref{e.cv}}{\leq} C_{\tilde{a}} \norm{u_1(t)}_{\cH^1} \norm{u_{1_0}}_{L^2(\T)}, 
\end{equation}
where in the last inequality we used that, for any solution $u_1(t,x)$ of \eqref{normal.instab} we have conservation of the $L^2$-norm of the solution, i.e., $\norm{u_1(t)}_{L^2(\T)} = \norm{u_{1_0}}_{L^2(\T)}$, for every $t \in \R$, due to skewadjointness of the operator $\Opw(i\xi (\e \la V \ra_m(x) + \e^{2} W(t,x)))$ (see \eqref{normal.instab}. 
Thus putting together \eqref{A.1} and \eqref{A.2} we have 
\begin{equation}\label{A.3}
    \norm{u_1(t)}_{\cH^1} \geq \left( C_{\tilde{a}} \norm{ u_{1_0}}_{L^2(\T)}\right)^{-1} e^{\e \delta_2 t } \left( \cA(0)- \frac{\beta_{\e, W, \tilde{a}}}{\delta_2} \norm{u_{1_0}}_{L^2(\T)}^2 \right).
\end{equation}
In the next paragraph we will show how to find $u_{1_0}\neq 0$ such that 
\begin{equation}\label{A.4}
    \left( \cA(0)- \frac{\beta_{\e, W, \tilde{a}}}{\delta_2} \norm{u_{1_0}}_{L^2(\T)}^2 \right)=:c_0 >0.
\end{equation}
This would conclude the proof of \eqref{growth.u1} for $s=1$ since using \eqref{A.4}, then \eqref{A.3} gives  \eqref{growth.u1}, with $\tilde\delta_1:= c_0 (C_{\tilde{a}} \norm{u_{1_0}}_{L^2(\T)})^{-1}>0$. Finally remark that $\norm{u}_s \geq \norm{u}_1$ for every $s\geq1$, thus exponential blow up of the Sobolev norm of the
solutions to \eqref{normal.instab} holds for any $s \geq 1$.
We now prove Lemma \ref{lemma.At}. 
\begin{proof}[Proof of Lemma \ref{lemma.At}]
    First of all remark that, computing the time derivative of $\cA(t)$ in \eqref{At} and using skewadjointness of $\Opw(i\xi( \e \la V \ra_m(x) + \e^2 W(t,x)))$ in \eqref{normal.instab}, we have
    \begin{equation}\label{dA}
        \dot{\cA}(t) = \la [ \Opw(-\tilde{a}), \Opw(i \xi (\e \la V \ra_m(x) + \e^2 W(t,x)))] u_1, u_1 \ra_{L^2(\T)}.
    \end{equation}
    Next, using results of symbolic calculus we have 
    
\begin{equation}
\begin{split}\label{b2}
    [ \Opw(-\tilde{a}), \Opw(i \xi (\e \la V \ra_m(x) + \e^2 W(t,x)))] & \stackrel{\eqref{comm.symbols}}{=} \Opw( \{ -\tilde{a},  \xi ( \e \la V \ra_m(x) + \e^2 W(t,x))\}) + \e \cB(t) \\
    & = \e \Opw(\{ \xi \la V \ra_m(x), \tilde{a} \}) + \e^2 \Opw(\{ \xi W(t,x) , \tilde{a} \} ) + \e \cB(t), 
\end{split}
    \end{equation}
    with $\cB(t) \in C^\infty(\T, \cS^{0})$ bounded remainder resulting from \eqref{comm.symbols}. Since $\tilde{a} \in \bS^1$, then $\{ \tilde{a}, \xi W(t,x) \} \in C^\infty(\T, \bS^1)$ (see \eqref{comm.symbols}). In particular, denoting by $M_W:= \max_{|\xi|=1} \left|\left\{\xi  W(x), \tilde{a}(x, \xi)\right\}\right| \geq 0$ one has
    \begin{equation}\label{b1}
     \Opw( \{ \xi W(t,x), a \}) \geq -M_W \Opw(|\xi|).
    \end{equation}
  Moreover, from  boundedness of $\cB(t)$ there exists $C_{\cB}>0$ such that 
    \begin{equation}\label{b.b}
        \la \cB(t) u_1, u_1 \ra_{L^2(\T)} \geq - C_{\cB} \norm{u_1}_{L^2(\T)}^2,  \quad \forall t \in \R. 
    \end{equation}
    Recalling again that $\norm{u_1(t,x)}_{L^2(\T)}= \norm{u_{1_0}(x)}_{L^2(\T)}$ and plugging \eqref{b.b}, \eqref{b1} and \eqref{b2} in \eqref{dA} we get
    \begin{equation}\label{b3}
        \dot{\cA}(t) \geq  \e \la \Opw( \{ \xi \la V \ra_m(x), \tilde{a} \} - \e M_W |\xi|) u_1, u_1 \ra_{L^2(\T)} - \e  C_{\cB} \norm{u_{1_0}}_{L^2(\T)}^2, \quad \forall t \in \R.
    \end{equation}
    Thus we are left to show that there exists $\delta_2, \tilde{C}_{\tilde{a}}>0$ such that 
    \begin{equation}\label{b4}
        \la \Opw( \{ \xi \la V \ra_m(x), \tilde{a} \} - \e M_W |\xi| )u_1, u_1 \ra_{L^2(\T)} \geq \delta_2 \cA(t)- \tilde{C}_{\tilde{a}} \norm{u_{1_0}}_{L^2(\T)}^2, \quad \forall t \in \R
    \end{equation} 
        Indeed, putting together \eqref{b4} and \eqref{b3}, we get \eqref{ineq.At} with $\beta_{\e, W, \tilde{a}}:= C_\cB +  \tilde{C}_{\tilde{a}}$. To prove \eqref{b4}, let $\chi_1 \in C_c^\infty(\T \times \R; [0, 1])$ be such that $\chi_1 = 1 $ if $|\xi|\leq 1$ and $\chi_1 =0$ if $|\xi| \geq 2$ and write 
        \begin{equation}
        \begin{split}
            \Opw(\{ \xi \la V \ra_m(x), \tilde{a} \} - \e M_W |\xi|) & = \Opw((\{ \xi \la V \ra_m(x), \tilde{a} \} - \e M_W |\xi|)(1-\chi_1))\\
            &+ \Opw((\{ \xi \la V \ra_m(x),\tilde{a} \}  - \e M_W |\xi|)\chi_1).
        \end{split}
        \end{equation}
        Remark that, by \eqref{tilde.a}, $\{ \xi \la V \ra_m(x), \tilde{a} \}(1-\chi_1)= \{ \xi \la V \ra_m(x), a \}(1-\chi_1)$. Moreover, from point 6 of Lemma \ref{symbolic.calculus} we have $\Opw((\{ \xi \la V \ra_m(x),\tilde{a} \} -\e M_W |\xi|) \chi_1) \in \cS^\rho$, for any $\rho <0$. In particular there exists $C_{\chi_1}>0$ such that 
    \begin{equation}\label{bb}
    \begin{split}
        \la \Opw( \{ \xi \la V \ra_m(x), \tilde{a} \} -\e M_W |\xi| )u_1, u_1 \ra_{L^2(\T)} & \geq \la \Opw((\{ \xi \la V \ra_m(x), a \}- \e M_W|\xi|) (1-\chi_1)) u_1, u_1 \ra_{L^2(\T)}\\
        & - C_{\chi_1} \norm{u_1}_{L^2(\T)}^2.
        \end{split}
    \end{equation}
        
        We exploit now \eqref{Poiss.esc} and we apply G\aa{}rding inequality:
        \begin{equation}\label{b5}
        \la \Opw( (\{ \xi \la V \ra_m(x), a \} -\e M_W|\xi|)(1-\chi_1)) u_1, u_1 \ra \stackrel{\substack{\eqref{Poiss.esc}\\ \eqref{eq.garding}}}{\geq} (\delta -\e M_W) \la \Opw((1-\chi_1)|\xi|)u_1, u_1 \ra -C_{G,1} \norm{u_1}_{L^2(\T)}^2, 
        \end{equation}
        where $C_{G,1}>0$ is the constant resulting from G\aa{}rding inequality and $\delta>0$ is given by \eqref{Poiss.esc}. Recalling now the definition of $\tilde{a}$ in \eqref{tilde.a} and using the fact that $a$ is positively homogeneous of degree one, there exists $C_{2, a}>0$ such that $a(x, \xi) \geq -C_{2,a} |\xi|$ for all $(x, \xi) \in \T \times \R$. Thus, applying again G\aa{}rding inequality, we obtain:  
        \begin{equation}\label{b6}
        \begin{split}
            \la \Opw((1-\chi_1)|\xi|) u_1, u_1 \ra & \geq -\frac{1}{C_{2,a}} \la\Opw( (1- \chi_1) a) u_1, u_1 \ra -C_{G,2} \norm{u_1}_{L^2(\T)}^2 \\
            & \geq -\frac{1}{C_{2,a}} \la \Opw( \tilde{a}) u_1, u_1 \ra - (C_{G,2} + C_{K}) \norm{u_1}_{L^2}^2, 
        \end{split}
        \end{equation}
        where $C_{G,2}>0$ is the constant resulting from G\aa{}rding inequality and $C_K>0$ results from the compact correction to pass from $\Opw((1- \chi_1)a)$ to $\Opw(\tilde{a})$.  

        Finally, plugging \eqref{b6} and \eqref{b5} in \eqref{bb} and choosing $0< \e < \frac{\delta}{2M_W}$ we obtain \eqref{b4} with $\tilde{C}_{\tilde{a}}=C_{G,1} + C_{G,2} + C_K + C_{\chi_1}>0$ and $\delta_2:=\frac{\delta}{2C_{2,a}}>0$ and concluding the proof. 
\end{proof}
\paragraph{Initial datum.} 
In this section we prove that there exists $u_{1_0} \in C^\infty(\T)$, such that $\norm{u_{1_{0}}}_{L^2(\T)}=1$ and
\begin{equation}\label{final}
    \la \Opw(-\tilde{a}) u_{1_0}, u_{1_0} \ra_{L^2(\T)} - \frac{\beta_{\e, W, \tilde{a}}}{\delta_2} >0,
\end{equation}
where the constants $\delta_2, \beta_{\e, W, \tilde{a}}$ are given by Lemma \ref{lemma.At}.
This gives \eqref{A.4} and, as discussed above, concludes the proof of the second point of Theorem \ref{main.thm}. 
To prove \eqref{final} we recall the property \eqref{pos.a} of the escape function, and we define $\chi_3 \in C_c^\infty(\T; [0,1])$, such that $\text{supp}(\chi_3) \subset \cW$, with $\cW$ given in \eqref{pos.a}. Next we define the family of functions 
\begin{equation}
    u_{\xi_0}(x):= \frac{\chi_3(x)}{\norm{\chi_3}_{L^2(\T)}} e^{i \xi_0 x}, \quad \forall x \in \T.
\end{equation}
We claim that there exist constants $C_\chi, C_r>0$ such that
\begin{equation}\label{done}
    \la \Opw(-\tilde{a}) u_{\xi_0}
, u_{\xi_0} \ra _{L^2(\T)}\geq C_\chi |\xi_0| - C_r.
\end{equation}
Thus choosing $\xi_0$ big enough, \eqref{done} proves \eqref{final} with $u_{1_0}=u_{\xi_0}$. To prove \eqref{done} first remark that, using symbolic calculus results in Lemma \ref{symbolic.calculus} we have 
\begin{equation}\label{c0}
     \la \Opw(-\tilde{a}) u_{\xi_0}
, u_{\xi_0} \ra _{L^2(\T)} \stackrel{\eqref{comp.symbols}}{=} \frac{1}{\norm{\chi_3}^2_{L^2(\T)}} \la \Opw(-\tilde{a} \chi_3^2) e^{i \xi_0 x}, e^{i \xi_0 x} \ra_{L^2(\T)} + \la \Opw(r_0)u_{\xi_0}, u_{\xi_0} \ra_{L^2(\T)} ,
\end{equation}
where $r_0 \in \bS^0(\T)$ results from the remainders in the pseudodifferential operations. Next, recalling the definition of $\tilde{a}$ in \eqref{tilde.a} and using \eqref{pos.a} we obtain 
\begin{equation}\label{c1}
\begin{split}
\la \Opw(-\tilde{a} \chi_3^2) e^{i \xi_0 x}
, e^{i \xi_0 x} \ra _{L^2(\T)} & 
\geq \la  \Opw\left( \frac{|\xi|}{2}(1-\chi)\chi_3^2  \right)  e^{i \xi_0 x}, e^{i \xi_0 x} \ra_{L^2(\T)} - C_g \\
& = \la \Opw \left(\frac{|\xi|}{2}(1-\chi) \right) \chi_3 e^{i \xi_0 x }, \chi_3 e^{i \xi_0 x} \ra_{L^2(\T)} - C_g + \la \Opw( r_1) e^{i \xi_0 x}, e^{i \xi_0 x} \ra_{L^2(\T)},
\end{split}
\end{equation}
where $C_g>0$ is the constant resulting from the G\aa{}rding inequality and where for the last equality we used again Lemma \ref{symbolic.calculus} and thus $r_1 \in \bS^0(\T)$ results from the remainders in the pseudodifferential operations. Finally, using the definition of Weyl quantization and performing a direct computation one obtains
\begin{equation}\label{cc}
    \la \Opw \left(\frac{|\xi|}{2}(1-\chi) \right) \chi_3 e^{i \xi_0 x }, \chi_3 e^{i \xi_0 x} \ra_{L^2(\T)} \stackrel{\eqref{tilde.a}}{\geq} 
    \int_{|\xi| \geq 1} |\xi| \left( \int_{\T^2} \chi_3(y) \overline{\chi_3(x)} e^{i(\xi - \xi_0)(x-y)}\di y \di x\right) \di \xi. 
\end{equation}
Using the Fourier series of $\chi_3$ given by $\chi_3(x)= \sum_{k \in \Z} c_k e^{ikx}$, one can check that 
\begin{equation}\label{delta.dir}
    \int_{\T^2} \chi_3(y) \overline{\chi_3(x)} e^{i(\xi - \xi_0)(x-y)}\di y \di x = (2\pi)^2 \sum_{k \in \Z} |c_k|^2 \delta(\xi- \xi_0 - k),
\end{equation}
where $\delta(\cdot)$ denotes the dirac-delta. Thus plugging \eqref{delta.dir} in \eqref{cc}, we obtain
\begin{equation}\label{last}
\begin{split}
    \la \Opw \left(\frac{|\xi|}{2}(1-\chi) \right) \chi_3 e^{i \xi_0 x }, \chi_3 e^{i \xi_0 x} \ra_{L^2(\T)}  & \geq (2\pi)^2 \sum_{k} |c_k|^2 \int_{|\xi| \geq 1} |\xi|  \delta(\xi - \xi_0 - k) \di \xi\\ 
    & = (2\pi)^2 \sum_{k} |c_k|^2 |\xi_0+ k| \\
    & \geq  (2\pi)^2 |\xi_0| \norm{\chi_3}_{L^2}^2 - C_0 \norm{\chi_3}_\frac12^2.
\end{split}
\end{equation}
 Thus putting \eqref{last} and \eqref{c1} in \eqref{c0} we obtain \eqref{done} with $C_\chi=(2\pi)^2$ and $C_r$ resulting from the sum of the constants $C_g$ and $C_0\norm{\chi_3}_{\frac12}$ and the remainders $r_0$ and $r_1$ (see \eqref{c1} and \eqref{c0}). This concludes the proof.

\section{Genericity and conclusion of the proof of Theorem \ref{main.thm}}\label{sec.generic}

We now conclude proving Point 3 of Theorem \ref{main.thm}. To this aim we first define the set 
\begin{equation}\label{R}
    \cR:= \{ f \in C^\infty(\T) : f(x)=0 \implies f'(x) \neq 0 \} \subset C^\infty(\T),
\end{equation}
of smooth non degenerate functions on $\T$ (see Definition \ref{non.deg}). We have the following property: 
\begin{lemma}\label{dense.R}
    The set $\cR$ is open and dense in $C^\infty(\T)$ endowed with its Fréchet topology (i.e., the one induced by the family of seminorms $(\norm{\cdot}_{C^k})_{k \geq 0}$). 
\end{lemma}

Let us postpone the proof of this fact and first use it to conclude the proof. To this aim, consider the application 
\begin{equation}
    \cF : C^\infty(\T^2) \to C^\infty(\T), \quad \cF(V(t,x)) = \la V \ra_m(x).
\end{equation}
One can verify that $\cF$ is linear and continuous. Moreover $\cV_{m, s} \cup \cV_{m, u}= \cF^{-1}(\cR)$. This immediately implies that $\cV_{m, s} \cup \cV_{m, u}$ is open. We now prove density of $\cV_{m, s} \cup \cV_{m, u}$. To this aim, take any $W \in C^\infty(\T^2)$, and consider its resonant average $\la W \ra_m(x)$ (recall Definition \ref{resonant.average}). From density of the set $\cR$ in \eqref{R}, for every $\e>0$ there exists $V_0 \in \cR$ such that $\td_{C^\infty(\T)}(V_0, \la W \ra_m) < \e$. Thus, considering the function 
\begin{equation}
    W_0(t,x):= V_0(x + mt) + ( W(t, x) - \la W \ra_m (x +mt)) 
\end{equation}
one finds that $\la W_0 \ra_m(x) = V_0(x) \in \cR$ by construction. Moreover, 
$$
\td_{C^\infty(\T^2)} ( W_0, W) = \td_{C^\infty(\T)} ( V_0, \la W \ra_m) < \e,
$$
concluding the proof. 

We now sketch the proof Lemma \ref{dense.R}.
\begin{proof}[Proof of Lemma \ref{dense.R}]
    We omit the proof of the fact that the set $\cR$ is open, since it is a direct application of Sard's Lemma and of the properties of the Fréchet topology (which ensures a control both on the $C^0$ and $C^1$ norms). We now prove that $\cR$ is dense in $C^\infty(\T)$: take any function $g \in C^\infty(\T) \setminus \cR$ (this means that $g$ has some degenerate zeroes). Using Sard's Lemma, for any $\e>0$, there exists $\e_0 \in [-\e, \e]$,regular value for $g$. Thus defining $h(x):= g(x) -\e_0$, we obtain that $\td_{C^\infty(\T)}(h, g)= \e_0< \e$. Moreover, the zeroes of $h(x)$ are the points such that $g(x)= \e_0$, thus $h'(x)= g'(x) \neq 0$, for all such $x \in \T$ since $\e_0$ is a regular value for $g$. Thus $h \in \cR$  concluding the proof. 
\end{proof}

\small

\bibliographystyle{plain}

\begin{thebibliography}{10}

\bibitem{Baldi}
P.~Baldi.
\newblock Periodic solutions of fully nonlinear autonomous equations of {Benjamin{\textendash}Ono} type.
\newblock {\em Annales de l'I.H.P. Analyse non lin\'eaire}, 30(1):33--77, 2013.

\bibitem{BBHM2018}
\newblock P. Baldi, M. Berti, E. Haus, R. Montalto.
\newblock Time quasi-periodic gravity water waves in finite depth. 
\newblock{\emph{Invent. Math.}} 214:739–911, 2018. 

\bibitem{BBM2014}
\newblock P. Baldi, M. Berti, R. Montalto. 
\newblock{ KAM for quasi-linear and fully nonlinear forced perturbations of Airy equation.} 
\newblock{ \emph{Math. Ann.}} 359(1–2):471–536, 2014. 

 
\bibitem{BGMR2017}
\newblock D. Bambusi, B. Grebert, A. Maspero, D. Robert.
\newblock Growth of Sobolev norms for abstract linear Schrödinger Equations. 
\newblock{\emph{JEMS}}, 2017.

\bibitem{BGMR}
D.~Bambusi, B.~Gr{\'e}bert, A.~Maspero, D.~Robert.
\newblock Reducibility of the quantum harmonic oscillator in d-dimensions with
  polynomial time-dependent perturbation.
\newblock {\em Analysis {\&} {PDE}}, 11(3):775--799, 2018.

\bibitem{BGMRV}
D. Bambusi, B. Grébert, A. Maspero, D. Robert, C. Villegas-Blas.
\newblock Longtime dynamics for the Landau Hamiltonian with a time dependent magnetic field.
\newblock  {\em EMS Surv. Math. Sci.} 12(1):155–184, 2025.


\bibitem{BLM.red.T}
\newblock D. Bambusi, B. Langella, R. Montalto.
\newblock Reducibility of Non-Resonant Transport Equation on $T^d$ with Unbounded Perturbations
\newblock {\em Annales Henri Poincaré}, 20:1893–1929, 2019.

\bibitem{Bellissard}
\newblock J. Bellissard. 
\newblock Stability and Instability in Quantum Mechanics. 
\newblock {\emph{Trends and developments in the eighties (Bielefeld. 1982/1983)}}  World Scientific, Singapore, 1-106, 1985.

\bibitem{Bourgain1999}
\newblock J. Bourgain. 
\newblock On growth of Sobolev norms in linear Schrödinger equations with smooth time dependent potential. 
\newblock{\emph{Journal d’Analyse Mathématique,}} 77(1):315–348, 1999.

\bibitem{Chabert}
\newblock A. Chabert. 
\newblock{ Weakly turbulent solution to Schrödinger equation on the two-dimensional torus with real potential decaying at infinity.} 
\newblock {\emph{\em Analysis {\&} {PDE}}, 18(8):2061–2080}, 2025.

\bibitem{Chabert2}
A. Chabert.
\newblock A Weakly Turbulent solution to the cubic Nonlinear Harmonic Oscillator on $\R^2$ perturbed by a real smooth potential decaying to zero at infinity.
\newblock{\em Comm PDE}, 49(3): 185-216, 2024.

\bibitem{Vlasov}
Y. Chaubet, D. Han-Kwan, G. Rivière.
\newblock Nonlinear chaotic Vlasov equations. \newblock{\em{ArXiv preprint, arXiv:2407.04426,} } 2024.

\bibitem{CKSTT}
J.~Colliander, M.~Keel, G.~Staffilani, H.~Takaoka, T.~Tao.
\newblock Transfer of energy to high frequencies in the cubic defocusing
  nonlinear {S}chr\"odinger equation.
\newblock {\em Invent. Math.}, 181(1):39--113, 2010.

\bibitem{DR.MS1}
N. V. Dang, G. Rivière.
  \newblock Spectral analysis of Morse-Smale flows I: construction of the anisotropic spaces. \newblock{\em{Journal de l'Institut de Math{\'e}matiques de Jussieu}}, 19(5):1409–1465, 2020. 

\bibitem{Colin_de_Verdi_re_2020}
Y.~C. de~Verdi{\`{e}}re.
\newblock Spectral theory of pseudodifferential operators of degree 0 and an application to forced linear waves.
\newblock {\em Analysis {\&} {PDE}}, 13(5):1521--1537, 2020.


\bibitem{CdVSR}
Y. C. de Verdière, L. Saint-Raymond.
\newblock Attractors for two‐dimensional waves with homogeneous hamiltonians of degree 0.
\newblock {\em CPAM}, 73(2):421--462, 2020.


\bibitem{Delort}
J.-M. Delort.
\newblock Growth of {S}obolev norms for solutions of time dependent
  {S}chr\"odinger operators with harmonic oscillator potential.
\newblock {\em Comm PDE}, 39(1):1--33, 2014.

\bibitem{FaouRaphael}
E. Faou, P. Raphaël.
\newblock On weakly turbulent solutions to the perturbed linear harmonic oscillator.
\newblock {\em American Journal of Mathematics}, 145(5): 1465--1507, 2023.

\bibitem{Faure.Sjostrand}
F. Faure, J. Sjöstrand.
\newblock Upper Bound on Density of Ruelle Resonances for Anosov Flows. \newblock{\em{Communications in Mathematical Physics}}, 308:325–364, 2011.

\bibitem{FeolaGiulianiMontaltoProcesi}
R. Feola, F. Giuliani, R. Montalto, M. Procesi.
\newblock Reducibility of first order linear operators on tori via Moser’s theorem. 
\newblock{\em{J. Funct. Anal.}}, 276(3):932-970, 2019. Corrigendum ibid. 279(2), Article ID 108542 (2020).


\bibitem{gerard_grellier}
P.~G\'erard, S.~Grellier.
\newblock The cubic {S}zeg{\H o} equation and {H}ankel operators.
\newblock {\em Ast\'erisque}, (389):vi+112, 2017.

\bibitem{GL}
P.~G\'erard,  E. Lenzmann.
\newblock{The Calogero--Moser Derivative Nonlinear Schr\"odinger Equation}.
\newblock{\em CPAM}, 77(10):4008-4062, 2024.

\bibitem{GG}
F. Giuliani, M. Guardia.
\newblock{Sobolev norms explosion for the cubic NLS on irrational tori}.
\newblock{\em Nonlinear Analysis}, 220,  2022. DOI:10.1016/j.na.2022.112865.


\bibitem{GHMMZ}
M. Guardia, E. Haus, Z. Hani, A Maspero, M. Procesi. 
\newblock Strong nonlinear instability and growth of
Sobolev norms near quasiperiodic finite-gap tori for the 2D cubic NLS equation.
\newblock {\em JEMS}, 25(4):1497–1551, 2022. 

\bibitem{guardia_kaloshin}
M. Guardia, V. Kaloshin. 
\newblock Growth of Sobolev norms in the cubic defocusing nonlinear Schr\"odinger 
equation. 
\newblock{\em JEMS}, 17(1):71--149, 2015.


\bibitem{hani14}
Z.~Hani.
\newblock Long-time instability and unbounded {S}obolev orbits for some
  periodic nonlinear {S}chr\"odinger equations.
\newblock {\em ARMA}, 211(3):929--964, 2014.

\bibitem{hani15}
Z.~Hani, B.~Pausader, N.~Tzvetkov, N.~Visciglia.
\newblock Modified scattering for the cubic {S}chr\"odinger equation on product
  spaces and applications.
\newblock {\em Forum of Mathematics, Pi}, 3:e4, 63, 2015.



\bibitem{HausMaspero}
E. Haus, A. Maspero.
\newblock{ Growth of Sobolev norms in time dependent semiclassical anharmonic oscillators}. 
\newblock{\em J.  Func.  Anal.} , 278(2), 108316, 2020.

\bibitem{Kuk}
S. Kuksin. 
\newblock{Growth and oscillations of solutions of nonlinear Schr\"odinger equation}.
\newblock{\em CMP, }178(2):265--280, 1996.

\bibitem{Kuk2}
S. Kuksin.  
\newblock{On turbulence in nonlinear Schr\"odinger equations.}
\newblock{\em GAFA,} 7(4):783--822, 1997.


\bibitem{LangellaMasperoRotolo25}
B. Langella, A. Maspero, M.T. Rotolo. 
\newblock  Growth of Sobolev norms for completely resonant quantum harmonic oscillators on $\mathbb{R}^2$.
 \newblock {\em J. Differ. Equations}, 433, 113221, 48 p., 2025.


\bibitem{lee2003introduction}
J.M. Lee.
\newblock {\em Introduction to Smooth Manifolds}.
\newblock Graduate Texts in Mathematics. Springer, 2003.

\bibitem{Lefeuvre}
T. Lefeuvre.
\newblock {\em Microlocal Analysis in Hyperbolic Dynamics and Geometry}.
\newblock Cours Sp\'ecialis\'es (SMF), to appear.

\bibitem{LLZ}
Z. Liang, J. Luo, Z. Zhao. 
\newblock{Symplectic Normal Form and Growth of Sobolev Norm. }
\newblock{\em J. Differ. Equations}, 449, 113702, 2025.

\bibitem{Mas18}
A. Maspero.
\newblock Lower bounds on the growth of Sobolev norms in some linear time dependent Schr{\"o}dinger equations.
\newblock {\em MRL}, 26(4):1197--1215, 2019.

\bibitem{Mas22}
A.~Maspero.
\newblock Growth of {S}obolev norms in linear {S}chr\"odinger equations as a dispersive phenomenon.
\newblock {\em Advances in Mathematics}, 411:Paper No. 108800, 48, 2022.


 \bibitem{Mas23}
A.~Maspero.
\newblock Generic transporters for the linear time-dependent quantum harmonic oscillator on {$\R$}.
\newblock {\em IMRN}, (14):12088--12118, 2023.

\bibitem{MasperoMurgante}
A. Maspero, F. Murgante.
\newblock One dimensional energy cascades in a fractional quasilinear NLS.
\newblock {\em arXiv}, 	arXiv:2408.01097,  2024.


\bibitem{MasperoRobert}
A. Maspero, D. Robert.
\newblock  On time dependent Schr{\"o}dinger equations: global well-posedness and growth of {Sobolev} norms.
 \newblock {\em{Journal of Functional Analysis}} 273(2):721-781, 2017.
 

\bibitem{Riviere.Rotolo}
G. Riviére, M. T. Rotolo.
\newblock{Sobolev Instability for Perturbation of Periodic Transport Equations}.
\newblock{\em{ArXiv preprint, arXiv:2510.17350}}, 2025.

\bibitem{Shubin}
M.A. Shubin.
  \newblock{\em Pseudodifferential Operators and Spectral Theory.}
  \newblock Springer Berlin Heidelberg, 2001. 

\bibitem{Taira}
K. Taira.
\newblock Equivalence of Classical and Quantum completeness for Real Principal Type Operators on the Circle.
\newblock {\em ArXiv preprint, arXiv:}2004.07547, 2024. 

\bibitem{Thomann}
L. Thomann.
\newblock Growth of Sobolev norms for linear Schr{\"o}dinger operators.
\newblock {\em Ann. H. Lebesgue}, 4:1595-1618, 2021.


\bibitem{Zworski}
M. Zworski.
\newblock{\em Semiclassical Analysis}. 
\newblock Graduate Studies in Mathematics. American Mathematical Soc., vol. 138, 2012. 

\end{thebibliography}

\end{document}